\documentclass[11pt,amstex]{article}  
\usepackage{color}              
\usepackage{graphicx}
\usepackage[]{amsmath}
\usepackage{amssymb}
\usepackage{hyperref}
\usepackage{latexsym, enumerate, amsfonts, amsthm, bbm}  

\newcommand*{\LargerCdot}{\raisebox{-0.25ex}{\scalebox{1.2}{$\cdot$}}}
\newcommand*{\Largertimes}{\raisebox{-0.25ex}{\scalebox{1.4}{$\times$}}}

\newtheorem{lemma}{Lemma}[section]
\newtheorem{proposition}[lemma]{Proposition}
\newtheorem{theorem}[lemma]{Theorem}

\newtheorem{remark}[lemma]{Remark}

\newtheorem{corollary}[lemma]{Corollary}
\newtheorem{definition}[lemma]{Definition}
\newtheorem{notation}[lemma]{Notation}

\newcommand{\EE}{\mathbb E}
\newcommand{\NN}{\mathbb N}
\newcommand{\PP}{\mathbb P}

\newcommand{\RR}{\mathbb R}

\newcommand{\SC}{\mathcal C}

\newcommand{\SF}{\mathcal F}
\newcommand{\SL}{\mathcal L}
\newcommand{\CP}{\mathcal P}

\newcommand{\SU}{\mathcal U}

\newcommand{\1}{\mathbbm{1}}
\newcommand{\norm}[1]{\parallel\! #1 \!\parallel_{\infty}}
\newcommand{\bignorm}[1]{\left|\!\left| #1 \right|\!\right|_{\infty}}
\newcommand{\normx}[2]{\parallel\! #1 \!\parallel_{ #2 }}
\newcommand{\cstar}{*}

\newcommand{\eqn}[1]{\begin{equation} #1 \end{equation}}
\newcommand{\eqan}[1]{\begin{align} #1 \end{align}}
\newcommand{\lbeq}[1]{\label{#1}}
\newcommand{\nn}{\nonumber}

\numberwithin{equation}{section}

\begin{document}

\author{
Sandra Kliem
\thanks{
Fakult\"at f\"ur Mathematik, Universit\"at Duisburg-Essen, Thea-Leymann-Str. 9, D-45127 Essen, Germany.
E-mail: {\tt sandra.kliem@uni-due.de}}}

\title{Travelling wave solutions to the KPP equation with branching noise arising from initial conditions with compact support}

\maketitle

\begin{abstract}
\noindent
We consider the one-dimensional KPP-equation driven by space-time white noise and extend the construction of travelling wave solutions arising from initial data $f_0(x) = 1 \wedge (-x \vee 0)$ from \cite{T1996} to non-negative continuous functions with compact support. As an application the existence of travelling wave solutions is used to prove that the support of any solution is recurrent. As a by-product, several upper measures are introduced that allow for a stochastic domination of any solution to the SPDE at a fixed point in time.
\end{abstract}

\vspace{0.3in}


\noindent
{\bf MSC2000 subject classification.} {Primary
60H15; 
Secondary
35R60. 
} 




%
%
\section{Introduction}
%
%
\subsection{Motivation}
%
%
Consider non-negative solutions to the one-dimensional stochastic partial differential equation (SPDE)
\eqan{
\lbeq{equ:SPDE} 
  & \partial_t u = \partial_{xx} u + \theta u - u^2 + u^{\frac{1}{2}} dW, \qquad t>0, x \in \RR, \theta>0 \\
  & u(0,x) = u_0(x) \geq 0, \nn
}
where $W=W(t,x)$ is space-time white noise and $\theta > 0$ a parameter. The deterministic part of this SPDE is (after appropriate scaling, cf. Mueller and Tribe~\cite[Lemma~2.1.2]{MT1994}) the well-studied Kolmogorov-Petrovskii-Piskunov-(KPP)-equation (also known as the Kolmogorov- or Fisher-equation). In Bramson~\cite{B1983} the existence of a family of non-negative travelling wave solutions to this deterministic partial differential equation (PDE) is established. Including the noise term, one can think of $u(t,x)$ as the density of a population in time and space. Leaving out the term $\theta u-  u^2$, the above SPDE is the density of a super-Brownian motion (cf. Perkins~\cite[Theorem~III.4.2]{bP2002}), the latter being the high density limit of branching particle systems that undergo branching random walks. The additional term of $\theta u$ models linear mass creation at rate $\theta > 0$, $-u^2$ models death due to overcrowding. In \cite{MT1995}, Mueller and Tribe obtain solutions to \eqref{equ:SPDE} as limits of densities of scaled long range contact processes with competition. The same techniques can be extended to obtain solutions to SPDEs with more general drift-terms, see Kliem~\cite{K2011}. 

The existence and uniqueness in law of solutions to \eqref{equ:SPDE} in the space of non-negative continuous functions with slower than exponential growth $\SC_{tem}^+$, is established in Tribe~\cite[Theorem~2.2]{T1996}. Let $\tau \equiv \inf\{ t \geq 0: u(t,\cdot) \equiv 0\}$ be the \textit{extinction-time} of the process. By \cite[Theorem~1]{MT1994}, there exists a critical value $\theta_c>0$ such that for any initial condition $u_0 \in \SC_c^+ \backslash \{0\}$ with compact support and $\theta < \theta_c$, the extinction-time of $u$ solving \eqref{equ:SPDE} is finite almost surely. For $\theta>\theta_c$, \textit{survival}, that is $\tau=\infty$, happens with positive probability. 

Let $R_0(u(t)) \equiv R_0(t) \equiv\sup\{ x \in \RR: u(t,x) > 0 \}$. Then $R_0(t)=-\infty$ if and only if $\tau \leq t$. Extending arguments of Iscoe~\cite{I1988} one can show that $R_0(u(0))<\infty$ implies $R_0(u(t))<\infty$ for all $t>0$. Using $R_0$ as a \textit{(right) wavefront marker}, we look for so-called \textit{travelling wave solutions} to \eqref{equ:SPDE}, that is solutions with the properties
\eqan{
  & (i) \qquad R_0(u(t)) \in (-\infty,\infty) \mbox{ for all } t \geq 0, \lbeq{equ:tv-i} \\
  & (ii) \quad\ \ u(t,\cdot + R_0(u(t))) \mbox{ is a stationary process in time}. \lbeq{equ:tv-ii}
}
In \cite[Section~3]{T1996} the existence of travelling wave solutions to \eqref{equ:SPDE} is shown, in \cite[Section~4]{T1996} it is established that for $\theta>\theta_c$ any travelling wave solution has an asymptotic (possibly random) \textit{wave speed}
\eqn{
  R_0(u(t))/t \rightarrow A \in \left[ 0,2\theta^{1/2} \right] \mbox{ for } t \rightarrow \infty \mbox{ almost surely.}
}
Strict positivity of $A$ remains an open problem if $\theta$ is not big enough. 

To construct a travelling wave, \cite{T1996} proceeds as follows. Use $R_1(u(t)) \equiv \ln\left( \int \exp(x) u(t,x) dx \right)$ in place of the wavefront marker $R_0(t)$ and take as initial condition $f_0(x) \equiv 1 \wedge (-x \vee 0)$ in \eqref{equ:SPDE}. Then the sequence $(\nu_T)_{T \in \NN}$ with
\eqn{
  \nu_T \ \mbox{ the law of } \ T^{-1}\!\int_0^T u(s,\cdot + R_1(u(s))) ds 
}
is tight and any limit point $\nu$ is nontrivial. Starting in $u_0$ with distribution $\nu$, shifted by $R_0$, one then obtains a travelling wave solution to \eqref{equ:SPDE}. 

The investigation of survival properties of solutions to \eqref{equ:SPDE} is a major challenge, where the main difficulty comes from the competition term $-u^2$. Without competition, the underlying ``additive property'' (cf. \cite[pages 167-168 and 159]{bP2002} in the context of Dawson-Watanabe superprocesses with drift) facilitates the use of Laplace functionals. Including competition, only subadditivity in the sense of \cite[Lemma~2.1.7]{MT1994} holds. 

It is quite common to first investigate the behaviour of solutions to SPDEs dependent on a parameter $\theta$ for $\theta$ very large respectively very small (see for instance \cite{MT1994} and \cite[Proposition~4.1c)]{T1996} in the context of \eqref{equ:SPDE}, Mueller and Sowers~\cite{MS1995} and Mueller, Mytnik and Quastel~\cite{MMQ2011} for KPP-type perturbed by a Fisher-Wright white noise). In the first case, to establish survival, a fruitful technique turns out to be comparison with $N$-dependent oriented site percolation with density $1-\rho$ (cf. Durrett~\cite[Chapter~4]{bD1995}). For $\theta$ big enough the influence of the stochastic part of the SPDE can be neglected on appropriately chosen time- and space- intervals and the linear drift of $\theta u$ dominates by far the influence of competition by $-u^2$. This allows for a first comparison with the solution to the corresponding PDE whose ability to ``generate and distribute mass'' forward in time is known. A second comparison of the latter with $N$-dependent percolation concludes the respective argument. Indeed, use that if the density of open sites is high enough, percolation occurs (cf. \cite[Theorem~A.1]{bD1995}).  In the second case where $\theta>0$ is small enough, the overall mass can be dominated by a stochastic process that goes extinct with probability one. Finally, comparison techniques can be used to show that the chance of survival is non-decreasing in $\theta$ (cf. \cite[Lemma~2.1.6]{MT1994}). In particular, the existence of $\theta_c>0$ follows. 

For $\theta>\theta_c$ close to criticality comparison with $N$-dependent oriented percolation is a difficulty, as competition increases dependence in space. Recall the construction of solutions to \eqref{equ:SPDE} in \cite{MT1995} by means of limits of densities of scaled long range contact processes with competition. For the nearest-neighbor contact process, Liggett~\cite[Theorem~2.28 of Chapter VI]{bL2005} gives a full description of the limiting law of a solution: for $\theta^{[contact]}>\theta_c^{[contact]}$, the limiting law is the weighted average of the Dirac-measure on the zero-configuration and the upper invariant measure of the process (cf. \cite[VI(1.1)]{bL2005}), where the weight on the former coincides with the extinction probability. A particularly interesting question is if such a complete convergence result holds true in the present setup? Note in particular that the result in \cite{bL2005} holds for any initial distribution. In Horridge and Tribe~\cite[Theorem~1]{HT2004} such a result is given for all $\theta>\theta_c$ under the assumption that the initial condition $u_0$ has property \cite[(6)]{HT2004}, that is, is ``uniformly distributed in space''. But what can be said for solutions to \eqref{equ:SPDE} with initial conditions with compact support when we condition on survival? Is a similar result valid? 

%
\subsection{Notation and background}
\label{ss:notation}
%
%
As a state space for solutions to \eqref{equ:SPDE} the \textit{space of non-negative continuous functions with slower than exponential growth} $\SC_{tem}^+$, defined as follows, is chosen. Let $\SC^+$ be the space of non-negative continuous functions on $\RR$, then
\eqn{
\lbeq{equ:Ctem}
  \SC_{tem}^+ = \left\{ f \in \SC^+: \normx{f}{\lambda} < \infty \mbox{ for all } \lambda>0 \right\} \mbox{ with } \normx{f}{\lambda} = \sup_{x \in \RR} |f(x)| e^{-\lambda |x|}.
}
Equip $\SC_{tem}^+$ with the topology given by the norms $\normx{f}{\lambda}$ for $\lambda>0$. Note that $d(f,g) \equiv \sum_{n \in \NN} (1 \wedge $ $\normx{f-g}{1/n})$ metrizes this topology and makes $\SC_{tem}^+$ a Polish space. Let $(\SC([0,\infty),\SC_{tem}^+),\SU,\SU_t,U(t))$ be continuous path space, the canonical right continuous filtration and the coordinate variables. 

Write $\langle f,g \rangle = \int f(x) g(x) dx$ and use $\Rightarrow$ to denote weak convergence of probability measures. Next, recall the following notation and result from \cite{T1996}.

\begin{notation}[equations (2.4)--(2.5) of \cite{T1996}] 
\label{NOTA:tribe}
Consider the generalized equation
\eqn{
\lbeq{equ:SPDE_T}
  \partial_t u = \partial_{xx} u + \alpha + \theta u - \beta u - \gamma u^2 + u^{\frac{1}{2}} dW
}
with $\alpha, \beta, \gamma \in \SC([0,\infty),\SC_{tem}^+)$. We may interpret $\alpha$ as the immigration rate, $\theta-\beta$ as the mass creation-annihilation rate and $\gamma$ as the overcrowding rate.

A solution to \eqref{equ:SPDE_T} consists of a filtered probability space $(\Omega,\SF,\SF_t,\PP)$, an adapted white noise $W$ and an adapted continuous $\SC_{tem}^+$ valued process $u(t)$ such that for all $\phi \in \SC_c^\infty$, the space of infinitely differentiable functions on $\RR$ with compact support,
\eqan{
  \langle u(t),\phi \rangle =& \langle u(0),\phi \rangle + \int_0^t \langle u(s),\phi_{xx}+(\theta-\beta(s)-\gamma(s)u(s)) \phi \rangle ds \\
  & + \int_0^t \langle \alpha(s),\phi \rangle ds + \int_0^t \int |u(s,x)|^{1/2} \phi(x) dW_{x,s}. \nn
}
If in addition $\PP(u(0,x)=f(x))=1$ then we say the solution $u$ starts at $f$.
\end{notation}

\begin{theorem}[Theorem~2.2a)--b) of \cite{T1996}] \hfill
\label{THM:tribe}
\begin{enumerate}
\item[a)] For all $f \in \SC_{tem}^+$ there is a solution to \eqref{equ:SPDE_T} started at $f$.
\item[b)] All solutions to \eqref{equ:SPDE_T} started at $f$ have the same law which we denote by $\PP^{f,\alpha,\beta,\gamma}$. The map $(f,\alpha,\beta,\gamma) \rightarrow \PP^{f,\alpha,\beta,\gamma}$ is continuous. The laws $\PP^{f,\alpha,\beta,\gamma}$ for $f \in \SC_{tem}^+$ form a strong Markov family.
\end{enumerate}
\end{theorem}

We now introduce additional notation used in the present article.

\begin{notation} \hfill
\label{NOTA:kliem}
\begin{enumerate}
\item
A solution to \eqref{equ:SPDE} is defined as in Notation~\ref{NOTA:tribe} above with $\alpha=\beta \equiv 0$ and $\gamma \equiv 1$. By Theorem~\ref{THM:tribe}, existence and uniqueness in law of solutions to \eqref{equ:SPDE} started at $u_0 \in \SC_{tem}^+$ hold and the laws $\PP_{u_0} \equiv \PP^{u_0,0,0,1}$ of such solutions form a strong Markov family on $\SC([0,\infty),\SC_{tem}^+)$. Use $\EE_{u_0}$ to denote respective expectations.
\item
In what follows denote $u_t(x) \equiv u(t,x)$, abbreviate $u_t \equiv u(t) \equiv u(t,\cdot)$ and write $u_t^{(u_0)}(x)$ for $u(t,x)$, starting at $u(0,x)=u_0(x)$. 
\item
Let $\tau \equiv \inf\{ t \geq 0: u(t,\cdot) \equiv 0\}$ be the \textit{extinction-time} of the process and denote the \textit{right wavefront marker} by
\eqn{
  R_0(f) \equiv \sup\{x \in \RR: f(x)>0 \} \mbox{ and write } R_0(t) \equiv R_0(u_t). 
}
For arbitrary $f \in \SC_{tem}^+ \backslash \{0\}$, $R_0(f) \in (-\infty,\infty]$. By the last two lines of \cite[Lemma~2.1]{T1996} (note that in the latter case $|x|$ has to be replaced by $x$; also see Remark~\ref{RMK:tribe_lemma_2_1} below) $R_0(0) < \infty$ implies $R_0(t) < \infty$ almost surely for all $t \geq 0$. Further adopt the obvious conventions $R_0(t)=-\infty$ on $\{ \tau \leq t \}$ and $u_t(\cdot+R_0(t)) \equiv 0$ on $\{\tau \leq t\}$. 

Analogously define the \textit{left wavefront marker} by $L_0(f) \equiv \inf\{x \in \RR: f(x)>0 \}$.
\item
For $\nu \in \CP(\SC_{tem}^+)$, the space of probability measures on $\SC_{tem}^+$, denote $\PP_\nu(A) \equiv \int_{\SC_{tem}^+} \PP_f(A) \nu(df)$. (Borel measurability in $f$ follows from the continuity of the map $f \mapsto \PP_f$ on $\SC_{tem}^+$, cf. Theorem~\ref{THM:tribe}b).)
\item
Use $\SC_c^+$ to denote the space of non-negative continuous functions with compact support and note that due to the \textit{compact support property}, $u_t \in \SC_c^+$ for all $t \geq 0$ if $u_0 \in \SC_c^+$. The compact support property follows for instance by reasoning as at the beginning of \cite[Section~2]{HT2004} or using \cite[Lemma~2.1]{T1996}.
\end{enumerate}
\end{notation}

Constants may change from line to line.
%
%
\subsection{Main results}
\label{ss:results}
%
%
The first main result of the present article is an alternative construction of travelling wave solutions in case $\theta>\theta_c$. The initial condition $f_0$ from \cite{T1996} is replaced by an arbitrary non-negative continuous function $g_0 \in \SC_c^+$ with compact support. As extinction (that is $\tau = \inf\{t \geq 0: u_t \equiv 0 \} = \inf\{t \geq 0: \langle u_t,1 \rangle = 0 \} < \infty$) happens with probability $0<\PP_{g_0}(\tau<\infty)<1$, we condition on non-extinction to obtain non-zero travelling wave solutions. 
%
%
\begin{definition} 
\label{DEF:nu_T}
For $\theta>\theta_c, g_0 \in \SC_c^+ \backslash \{0\}$ let $\nu_T \in \CP(\SC_{tem}^+)$ be given by
\eqn{	
  \nu_T(A) \equiv T^{-1}\!\int_0^T \PP_{g_0}(u_s(\cdot+R_0(s)) \in A \;|\; \tau=\infty) ds,
}
that is, $\nu_T$ is the law of $T^{-1}\!\int_0^T u_s(\cdot+R_0(s)) ds$ under $\PP_{g_0}(\cdot \;|\; \tau=\infty) \in \CP(\SC([0,\infty),\SC_{tem}^+))$.
\end{definition}
%
%
\begin{remark} 
\label{RMK:def_nu_T}
For $\theta>\theta_c, g_0 \neq 0$ this yields in particular (recall Notation~\ref{NOTA:kliem}-4) $\PP_{\nu_T} \in \CP(\SC([0,\infty),\SC_{tem}^+))$ with
\eqn{	
\lbeq{equ:def_nu_T}
  \PP_{\nu_T}(B)
    = (\PP_{g_0}(\tau=\infty))^{-1} T^{-1} \int_0^T \EE_{g_0}\!\left[ \1_{\{\tau=\infty\}} \PP_{u_s(\cdot+R_0(s))}(B) \right] ds,
}
where $0< \PP_{g_0}(\tau=\infty) <1$.
\end{remark}
%
%
The sequence $(\nu_T)_{T \in \NN}$ is in the sequel shown to be tight for $\theta>\theta_c, g_0 \in \SC_c^+ \backslash \{0\}$ fixed. Every subsequential limit $\nu$ yields the (not necessarily unique) law $\PP_\nu$ of a travelling wave solution, that is under $\PP_\nu$, \eqref{equ:tv-i} and \eqref{equ:tv-ii} hold. 
%
%
\begin{theorem}
\label{THM:trav_wave_exists}
Let $\theta>\theta_c$ and $g_0 \in \SC_c^+ \backslash \{0\}$. Every subsequential limit of the tight sequence $\{ \nu_T: T \in \NN \}$ from Definition~\ref{DEF:nu_T} yields a travelling wave solution to equation \eqref{equ:SPDE}. 
\end{theorem}
%
%
We work with the original wavefront marker $R_0(t)$ and show in particular that the wavefront marker of the limiting solution is zero with probability one. 
%
%
\begin{proposition} \label{PRO:limit_never_zero}
Let $\theta>\theta_c$, $g_0 \in \SC_c^+ \backslash \{0\}$ and let $\nu_{T_n}$ be a subsequence that converges to $\nu$. Then $\nu(\{ f: R_0(f) = 0 \})=1$ and $\PP_\nu(u(t) \not \equiv 0)=1$ for all $t \geq 0$.
\end{proposition}
%
%

The uniqueness of the law of travelling wave solutions remains an open problem. Is the wavespeed deterministic and/or does it depend on $f_0, g_0$? In how far does it depend on $\theta>\theta_c$ and is it strictly positive (for $\theta$ big enough, cf. \cite[Proposition~4.1c)]{T1996})? In the latter case, does the same hold true for $\limsup_{t \rightarrow \infty} R_0(u(t))/t$? 

Survival is possible only if the overall mass and $R_0(t)-L_0(t)$ with $L_0(u(t)) \equiv L_0(t) \equiv\inf\{ x \in \RR: u(t,x) > 0 \}$ grow to infinity over time (see Proposition~\ref{PRO:overall_mass_and_support} below). Recall the definitions of \textit{recurrence, transience} and \textit{local extinction} of the support of a process from Pinsky~\cite[Definitions, p. 239--240]{P1996}. Here we formulate the definition of \textit{recurrence} in the context of SPDEs.
\begin{definition}
\label{DEF:recurrence}
The support of the process $u(t)$ is \textit{recurrent} if 
\eqn{
\lbeq{equ:def-recurrent}
  \PP_{u_0}\!\left( \int u(t,x) \1_B(x) dx > 0 \mbox{ for some } t \geq 0 \ \big| \ \tau=\infty \right) = 1
}
for every $u_0 \in \SC_c^+ \backslash \{0\}$ and every open set $B \subset \RR$.
\end{definition}
A consequence of the existence of travelling waves constructed from compact initial conditions is our second main result which shows that the solutions to the SPDE \eqref{equ:SPDE} have recurrent support. 
%
%
\begin{theorem}
\label{THM:coming_back}
Let $\theta>\theta_c$ and $g_0 \in \SC_c^+ \backslash \{0\}$. Then the support of the process $u(t)$ is recurrent.
\end{theorem}
%
%
We obtain as a by-product of the proof that the supremum of the process does not decrease to zero over time almost surely (see Lemma~\ref{LEM:for-all-outside-dom}). The idea of the proof is that ``a travelling wave comes back''. By monotonicity one obtains similar results for any initial condition in $\SC_{tem}^+$. This result is a first step in the direction of obtaining a more detailed view on the behaviour of surviving solutions to \eqref{equ:SPDE}, in particular in near-critical regions $\theta>\theta_c$ where no results except for \cite[Theorem~1]{HT2004} are known to the knowledge of the author. It remains an open problem to give a complete convergence result. 

The present article additionally introduces upper measures that allow for a stochastic domination of solutions to \eqref{equ:SPDE} starting in initial conditions of arbitrary support respectively support bounded to the left or right (cf. Remark~\ref{RMK:on_u_star_l}). In the first case we obtain a second construction for the unique translation invariant stationary distribution in the convergence result of \cite[Theorem~1]{HT2004}.
%
%
\subsubsection{Comparison with the construction in \cite{T1996}}
%
%
In \cite[Section~3]{T1996}, $f_0(x)=1 \wedge (-x \vee 0) \in \SC_{tem}^+$ is fixed as initial condition. Note that for $\theta>\theta_c$, $\PP_{f_0}(\tau=\infty)=1$, so there is no need to condition on survival. Instead of the wavefront marker $R_0(t)$, $R_1(t)=R_1(u(t)) \equiv \ln\left( \int \exp(x) u(t,x) dx \right)$ is used. Taking these changes into account, $\nu_T^{\tiny \cite{T1996}}$ and $\PP_{\nu_T}^{\tiny \cite{T1996}}$ are constructed as in Definition~\ref{DEF:nu_T} and Remark~\ref{RMK:def_nu_T} above. For $\theta>\theta_c$, tightness of the sequence $(\nu_T)^{\tiny \cite{T1996}}_{T \in \NN}$ is established in \cite[Lemma~3.7]{T1996}. By \cite[Theorem~3.8]{T1996} every subsequential limit $\nu^{\tiny \cite{T1996}}$ yields the law $\PP_{\nu^{\tiny \cite{T1996}}}$ of a travelling wave solution. 

As remarked in \cite{T1996}, the proof that the limit point is non-trivial seems easier using the wavefront marker $R_1$ rather than $R_0$. To establish the recurrence result in Theorem~\ref{THM:coming_back}, the use of the wavefront marker $R_0$ is more suited. As a result, a substantial part of the work to follow goes into establishing a result in the spirit of \cite[Lemma~3.5]{T1996}. By means of domination methods it is shown that the slope of the corresponding linear function only depends on $\theta$ but not on the initial condition $g_0 \in \SC_c^+$. Due to the use of the wavefront marker $R_0$, an additional argument becomes necessary to obtain the properties of any subsequential limit $\nu$ as detailed in Proposition~\ref{PRO:limit_never_zero}. 

%
%
\subsection{Outline}
%
%
The paper is organized as follows. In Section~\ref{SEC:self_duality_and_uim} useful technical properties are recalled and upper measures that allow for a stochastic domination of solutions to \eqref{equ:SPDE} are obtained. In particular, the unique translation invariant stationary distribution in the convergence result of \cite[Theorem~1]{HT2004} is obtained as the unique weak limit of a sequence of dominating measures. In Section~\ref{SEC:mass_and_support}, the blow up of the overall mass and support of a solution conditional on survival is established. Estimates on the right wavefront marker are derived in Section~\ref{SEC:estimates_wfm}. They are used in Section~\ref{SEC:construction_tw} to construct travelling wave solutions arising from initial conditions with compact support. We show in particular that the wavefront marker of the limiting solution is zero with probability one. Finally, in Section~\ref{SEC:recurrence}, the recurrence of the support of a solution conditional on survival is shown.
%
%
\section{Self duality and upper measures}
\label{SEC:self_duality_and_uim}
%
%

As detailed in \cite[Section~1.2]{HT2004}, the following \textit{self duality} relation for \eqref{equ:SPDE} holds: Let $u, v$ be independent solutions to \eqref{equ:SPDE} with $u_0 \in \SC_{tem}^+, v_0 \in \SC_c^+$. Then we have for $0 \leq s \leq t$,
\eqn{
\lbeq{equ:self-duality}
  \EE\!\left[ e^{-2 \langle u(t),v(0) \rangle} \right] = \EE\!\left[ e^{-2 \langle u(s),v(t-s) \rangle} \right]  = \EE\!\left[ e^{-2 \langle u(0),v(t) \rangle} \right]. 
}
We extend this to arbitrary $v_0 \in \SC_{tem}^+$ as outlined below. 

First approximate $v_0$ by a monotonically increasing sequence $\big\{ v_0^{(n)} \big\}_{n \in \NN}$ in $\SC_c^+$ such that $v_0^{(n)} \uparrow v_0$. Adapt the reasoning of \cite[(12)--(13)]{T1996}, based on a coupling technique of Barlow, Evans and Perkins~\cite{BEP1991}, inductively as follows. For $n=1$, let $u$ be a solution to \eqref{equ:SPDE} started at $u_0=v_0^{(1)}$ and defined on $(\Omega,\SF,\SF_t,\PP)$. Define a random $\SC_{tem}^+$-valued process by
\eqn{
\lbeq{equ:coupling-B}
  B(\omega)(t,x) = 2 u(\omega)(t,x).
}
Let (recall the notations $(\SC([0,\infty),\SC_{tem}^+),\SU,\SU_t,U(t))$ from below \eqref{equ:Ctem} and $\PP^{f,\alpha,\beta,\gamma}$ from Theorem~\ref{THM:tribe})
\eqan{
  & \Omega^{(2)} = \Omega \times \SC([0,\infty),\SC_{tem}^+), \quad \SF^{(2)} = \SF \times \SU, \quad \SF_t^{(2)} = \SF_t \times \SU_t \\
  & u^{(1)}(\omega,f) = u(\omega), \quad v(\omega,f) = U(f), \quad w(\omega,f) = u^{(1)}(\omega,f) + v(\omega,f). \nn
}
Then there is a unique probability $\PP^{(2)}$ on $(\Omega^{(2)},\SF^{(2)})$ such that for $F \in \SF, G \in \SU$,
\eqn{
\lbeq{equ:coupling-BEP}
  \PP^{(2)}(F \times G) = \int_\Omega \1_F(\omega) \;\PP^{\big( v_0^{(2)}-v_0^{(1)} \big),0,B(\omega),1}(G) \;\PP(d\omega).
}
The integrand on the right hand side is measurable by the continuity of the map $(f,\alpha,\beta,\gamma) \mapsto \PP^{f,\alpha,\beta,\gamma}$ (cf. Theorem~\ref{THM:tribe}). The techniques of \cite[Theorem~5.1]{BEP1991} show that $w$ solves \eqref{equ:SPDE} with $w_0=v_0^{(2)}$ (on a possibly enlarged probability space where $\Omega^{(2)}$ is replaced by $\Omega^{(2)} \times \SC([0,\infty),\SC_{tem}^+)$). The idea for the construction of $v$ is to add an independent white noise (to that for $u$) and to obtain $v$ (conditional on $u$) as a process with annihilation due to competition with $u$. To establish the existence of a white noise for $w$, a solution to \eqref{equ:SPDE} (cf. Notation~\ref{NOTA:tribe}), one may have to enlarge the probability space.

Now proceed inductively. For $n \geq 2$, let $u_0=v_0^{(n)}$, $(\Omega,\SF,\SF_t,\PP)=(\Omega^{(n)},\SF^{(n)},\SF_t^{(n)},\PP^{(n)})$ and use $ \PP^{\big( v_0^{(n+1)}-v_0^{(n)} \big),0,B(\omega),1}$ in the integrand of \eqref{equ:coupling-BEP} in the $n$-th step of the construction. By Ionescu-Tulcea's theorem (see for instance Klenke \cite[Theorem~14.32]{bK2014}) there exists a uniquely determined probability measure $\PP'$ on $\Omega \times (\SC([0,\infty),\SC_{tem}^+)^\NN$ such that for $F \in \SF, G \in \SU^{\otimes n}$,
\eqn{
\lbeq{equ:coupling-IT}
  \PP'\big(F \times G \times \big( \Largertimes_{i = n+1}^\infty \SC([0,\infty),\SC_{tem}^+) \big) \big) = \PP^{(n+1)}(F \times G).
}
Use the above inductive construction to obtain a coupled sequence of solutions $u^{(n)}$ to \eqref{equ:SPDE} on $\big( \Omega \times (\SC([0,\infty),\SC_{tem}^+)^\NN, \SF \times \SU^\NN, \PP' \big)$ with initial conditions $u_0^{(n)} = v_0^{(n)}$ satisfying $u_t^{(n)}(x) \leq u_t^{(n+1)}(x)$ for all $t \geq 0, x \in \RR$. Let $\overline{u}_t \equiv \uparrow \lim_{n \rightarrow \infty} u_t^{(n)}$. By \eqref{equ:coupling-IT} and the tightness of the sequence $\PP_{v_0^{(n)}}$ for $n \rightarrow \infty$, $\overline{u}$ has law $\PP_{v_0}$ (cf. Theorem~\ref{THM:tribe}). Equality in \eqref{equ:self-duality} then follows by using monotone convergence for each term separately. 

Note that we do not claim that $\overline{u}$ is a solution to \eqref{equ:SPDE} as we did not prove the existence of an appropriate adapted white noise for this limit. The uniqueness in law is sufficient for the required result \eqref{equ:self-duality}.
%
%
\begin{remark}[monotonicity and domination by a superprocess]
\label{RMK:monotonicity-domination_super}
In what follows we frequently make use of the following two properties. Note in particular the discussion at the beginning of \cite[Section~2]{HT2004}.
\begin{enumerate}
\item[(i)] \textit{(monotonicity)} \\
Let $u_0, v_0 \in \SC_{tem}^+$ satisfy $u_0(x) \leq v_0(x)$ for all $x \in \RR$. Reason as above (use $\PP^{v_0-u_0,0,B(\omega),1}$ in \eqref{equ:coupling-BEP}) to see that on a common probability space, there exist solutions $u(t,x)$ and $v(t,x)$ to \eqref{equ:SPDE} with initial conditions $u_0$ respectively $v_0$ such that $u(t,x) \leq v(t,x)$ for all $t \geq 0, x \in \RR$ almost surely. Once again, the idea is to construct the increment $v-u$ as a process with annihilation due to competition with $u$ by means of an independent (of $u$) white noise. 

Next consider initial conditions $\underline{u}_0, \overline{u}_0^{(i)} \in \SC_{tem}^+, i \in \NN$ such that $\underline{u}_0(x) \leq \sum_{i \geq 1} \overline{u}_0^{(i)}(x)$ for all $x \in \RR$. Without loss of generality decompose $\underline{u}_0 = \sum_{i \geq 1} \underline{u}_0^{(i)}$ with $\underline{u}_0^{(i)} \in \SC_{tem}^+, i \in \NN$ satisfying $\underline{u}_0^{(i)}(x) \leq \overline{u}_0^{(i)}(x)$ for all $i \in \NN$. Then, on a common probability space, one can construct solutions $\underline{v}^{(I)}, \overline{u}^{(i)}, I, i \in \NN$ to \eqref{equ:SPDE} with initial conditions $\underline{v}_0^{(I)} \equiv\sum_{1 \leq i \leq I} \underline{u}_0^{(i)}, I \in \NN$ respectively $\overline{u}_0^{(i)}, i \in \NN$ such that 
\eqn{
\lbeq{equ:coupling-NN}
  \underline{u}_t \equiv \uparrow \lim_{I \rightarrow \infty} \underline{v}_t^{(I)} \ \mbox{ has law } \ \PP_{\underline{u}_0} \ \mbox{ and satisfies } \ \underline{u}_t(x) \leq \sum_{i \geq 1} \overline{u}_t^{(i)}(x) 
}
for all $t \geq 0, x \in \RR$ almost surely.

Indeed, we shortly outline how to extend the construction of \eqref{equ:coupling-B}--\eqref{equ:coupling-IT} to this case. Using monotonicity, let $\underline{v}^{(1)}, \overline{u}^{(1)}$ be a coupled pair of solutions on some probability space $(\Omega,\SF,\SF_t,\PP)$ to \eqref{equ:SPDE} started at $\underline{u}_0^{(1)}$ respectively $\overline{u}_0^{(1)}$ such that $\underline{v}^{(1)}_t(x) \leq \overline{u}^{(1)}_t(x)$ for all $t \geq 0, x \in \RR$ almost surely. Inductively construct solutions $\underline{v}^{(n+1)}, \overline{u}^{(n+1)}$ on $\Omega^{(n+1)} = \Omega \times (\SC([0,\infty),\SC_{tem}^+))^{2n}$, $n \in \NN$ by considering 
\eqan{
\lbeq{equ:coupling-NN-constr}
  u^{(n)}(\omega,f,g) = \underline{v}^{(n)}(\omega), \quad & \underline{v}(\omega,f,g) = U(f), \quad \underline{w}(\omega,f,g) = u^{(n)}(\omega,f,g) + \underline{v}(\omega,f,g), \\
  & \overline{v}(\omega,f,g) = U(g), \quad \overline{w}(\omega,f,g) = \underline{v}(\omega,f,g) + \overline{v}(\omega,f,g) \nn
}  
and for $F \in \SF$ and $G, H \in \SU$,
\eqan{
  & \PP^{(n+1)}(F \times G \times H) \\
  &= \int_{\Omega^{(n)}} \1_F(\omega) \int_{\SC_{tem}^+} \1_G(f) \;\PP^{\big( \overline{u}_0^{(n+1)} - \underline{u}_0^{(n+1)} \big),\overline{A}(\omega,f),\overline{B},f(\omega),1}(H) \;\PP^{\underline{u}_0^{(n+1)},0,\underline{B}(\omega),1}(df) \;\PP(d\omega) \nn
}
with the three random $\SC_{tem}^+$-valued processes $\underline{B}, \overline{A}$ and $\overline{B}$ given as
\eqan{
  & \underline{B}(\omega)(t,x) = 2 \underline{v}^{(n)}(\omega)(t,x), \\
  & \overline{A}(\omega,f)(t,x) = 2 \underline{v}^{(n)}(\omega)(t,x) f(t,x) \ \mbox{ and } \ \overline{B}(\omega,f)(t,x) = 2 f(t,x). \nn
}
Note that in this step, $\underline{w}, \overline{w}$ solve \eqref{equ:SPDE} with $\underline{w}_0 = \underline{v}_0^{(n+1)}$ respectively $\overline{w}_0 = \overline{u}_0^{(n+1)}$. The claim in \eqref{equ:coupling-NN} now follows from \eqref{equ:coupling-NN-constr} and $\underline{v}^{(1)}_t(x) \leq \overline{u}^{(1)}_t(x)$.
\item[(ii)] \textit{(domination by a superprocess for $\theta>0$)} \\
Recall from \cite[(7)]{HT2004} that there is a coupling of a solution $u$ to \eqref{equ:SPDE} starting in $u_0 \in \SC_{tem}^+$ with a solution $\bar{u}$ to 
\eqn{
  \partial_t \bar{u} = \partial_{xx} \bar{u} + \theta \bar{u} + \bar{u}^{\frac{1}{2}} dW, \quad \bar{u}_0=u_0
}
so that $u(t,x) \leq \bar{u}(t,x)$ for all $t \geq 0, x \in \RR$ almost surely. $\bar{u}$ is the density of a one-dimensional Dawson-Watanabe superprocess with constant mass creation $\theta$ (cf. \cite{bP2002}) and (use the notations $\overline{\PP}_{u_0}$ and $\bar{\tau}$ in what follows to indicate the use of a coupling with an appropriate $\bar{u}$)
\eqn{
\lbeq{equ:bound_die_out_t}
  \PP_{u_0}(\tau \leq t) \geq \overline{\PP}_{u_0}(\bar{\tau} \leq t) = \exp\!\left( \frac{-2\theta \langle u_0 , 1 \rangle}{1-e^{-\theta t}} \right)
}
(cf. \cite[Exercise~II.5.3]{bP2002}). Furthermore, reason as in the proof of \cite[(10)]{HT2004} to get for arbitrary $-\infty<a<b<\infty$ and $t>0$,
\eqn{
\lbeq{equ:bound_die_out_interval_t}
  \inf_{u_0 \in \SC_c^+: \mathrm{supp}(u_0) \subset [a,b]} \PP_{u_0}\!\left( \tau \leq t \right)
  \geq C(\theta,t,a,b)
  >0.
}
\end{enumerate}
\end{remark}
%

The main result of \cite{HT2004} is that for $\theta>\theta_c$ there exists a unique translation-invariant stationary measure $\mu \in \CP(\SC_{tem}^+)$ with $\mu(\{ f: f \not\equiv 0 \}) = 1$ that is a stationary distribution for \eqref{equ:SPDE}. They give sufficient conditions for its domain of attraction in \cite[(6)]{HT2004}. Note that in \cite[(6)]{HT2004} the condition is uniform in $x$, where $\{ T_t: t \geq 0\}$ denotes the heat semigroup, and thus does not extend to compact initial conditions but includes for instance all positive constant functions. See the paragraph below \cite[(4)]{HT2004} for a motivation of the proof of this theorem. $\mu$ is further characterized by its Laplace functional
\eqn{
\lbeq{equ:upper-invariant-Laplace}
  \int e^{-2 \langle f , g \rangle} \mu(df) = \PP_g(\tau<\infty), \quad \mbox{ for } g \in \SC_c^+.
}
%

In what follows we give a second construction of $\mu$ as the unique weak limit of a sequence $(\mu_T)_{T>0}$ for $T \rightarrow \infty$, where $\mu_T$ can be thought of as a dominating measure at time $T$ to any solution to \eqref{equ:SPDE} with $u_0 \in \SC_{tem}^+$ (look ahead to Corollary~\ref{COR:coupling_dominated_by_upper} for a precise statement of this result).
%
%
\begin{proposition} 
\label{PRO:tightness_T}
Let $T>0$ be fixed and $\psi_N \in \SC_{tem}^+$ such that $\psi_N(x) \uparrow \infty$ as $N \rightarrow \infty$ for all $x \in \RR$. Then $\SL\!\left( u_T^{(\psi_N)} \right) \Rightarrow \mu_T$ for $N \rightarrow \infty$ in $\CP\!\left( \SC_{tem}^+ \right)$, where $\mu_T$ is uniquely characterized by its Laplace functional
\eqn{
\lbeq{equ:Laplace-u-T}
  \int e^{-2 \langle f,g \rangle} \mu_T(df) = \PP\!\left( \tau^{(g)} \leq T \right), \quad \mbox{ for all } g \in \SC_{tem}^+
}
with
\eqn{
  \tau^{(g)} \equiv \inf\!\left\{ t \geq 0: u_t^{(g)} \equiv 0 \right\}. 
}
\end{proposition}
%
%
\begin{proof}
The Kolmogorov tightness criterion for a sequence of laws of $\SC_{tem}^+$ valued random variables $(X_N)_{N \in \NN}$ is stated in \cite[(2) and below]{T1996}: it is sufficient to show that
\begin{enumerate}
\item[(i)] $\left\{ \SL\!\left( \langle X_N,e^{-|\cdot|} \rangle \right) \right\}_{N \in \NN}$ is tight,
\item[(ii)] for all $\lambda>0$ there exist $C<\infty, p>0, \gamma>1, \xi<\lambda$ such that for all $N \in \NN$,
\eqn{
  \EE\!\left[ |X_N(x)-X_N(x') |^p \right] \leq C |x-x'|^\gamma e^{\xi p |x|} \mbox{ for all } |x-x'| \leq 1.
}
\end{enumerate}

To check condition (i) first let
\eqn{
\lbeq{cond_i_1}
  \phi(x) \equiv e^{1-\frac{1}{2}\sqrt{1+x^2}} \geq e^{-|x|} 
}
so that $\phi \in \SC^2$, $\phi>0$. Apply \cite[Lemma~3.3]{T1996} in what follows. First observe that $\phi \in \Phi \equiv \{ f: \normx{f}{\lambda} < \infty$ for some $\lambda<0 \}$, the space of functions with exponential decay, where $\normx{f}{\lambda} = \sup_x |f(x)| \exp(-\lambda|x|)$. Moreover,
\eqn{
  \phi_x(x) = -\frac{\phi(x)}{2} \frac{x}{\sqrt{1+x^2}} \mbox{ and } \phi_{xx}(x) = \phi(x) \left( \frac{x^2}{4(1+x^2)} - \frac{1}{ 2 (1+x^2)^{3/2} } \right) \in \Phi 
}
and 
\eqan{
  \alpha \equiv
  & \bignorm{ \frac{\phi_{xx}(x)}{\phi(x)} } \leq \frac{1}{4} \bignorm{ \frac{x^2}{1+x^2} } + \frac{1}{2} \bignorm{ \frac{1}{ (1+x^2)^{3/2} } } \leq 1 < \infty, \\
  \beta \equiv& \norm{\phi} = e^{1/2}, \nn\\
  \gamma \equiv& \langle \phi,1 \rangle \leq \langle e^{1-\frac{1}{2}|\cdot|},1 \rangle = 2 e \int_0^\infty e^{-x/2} dx = 4e < \infty. \nn
}
Hence \cite[Lemma~3.3]{T1996} yields for all $p \geq 2, N \in \NN$ that
\eqn{
  \EE\!\left[ \left( \big\langle u_T^{(\psi_N)},e^{-|\cdot|} \big\rangle \right)^p \right] \leq \EE\!\left[ \left( \big\langle u_T^{(\psi_N)},\phi \big\rangle \right)^p \right] \leq C(\theta,p,T)
}
and we obtain
\eqn{
\lbeq{cond_i_2}
  \PP\!\left( \big\langle u_T^{(\psi_N)},e^{-|\cdot|} \big\rangle \geq K \right) 
  \leq \PP\!\left( \big\langle u_T^{(\psi_N)},\phi \big\rangle \geq K \right) 
  \leq \frac{1}{K^p} \EE\!\left[ \left( \big\langle u_T^{(\psi_N)},\phi \big\rangle \right)^p \right] \leq \frac{C(\theta,p,T)}{K^p} \stackrel{K \rightarrow \infty}{\rightarrow} 0. 
}

\cite[Lemma~3.4]{T1996} yields for all $\theta > 0, p \geq 2$ and $|x-x'| \leq 1$,
\eqn{
\lbeq{equ:bound_on_p_diff}
  \EE\!\left[ \left| u^{(\psi_N)}_T(x)-u^{(\psi_N)}_T(x') \right|^p \right] \leq C(\theta,p,T) |x-x'|^{p/2-1}.
}
Choose $p>4$ to obtain (ii) with $C=C(\theta,p,T), \gamma \equiv p/2-1>1$ and $\xi \equiv 0<\lambda$ for $\lambda>0$ arbitrary. It follows that $\left( \SL\!\left( u_T^{(\psi_N)} \right) \right)_{N \in \NN}$ is tight in $\CP\!\left( \SC_{tem}^+ \right)$.

Finally we show that $\SL\!\left( u_T^{(\psi_N)} \right) \stackrel{N \rightarrow \infty}{\Rightarrow} \mu_T$. To this goal we investigate the Laplace functional of the subsequential limits of $\SL\!\left( u_T^{(\psi_N)} \right)$: Let $(N_k)_k$ be a subsequence such that $\SL\!\left( u_T^{(\psi_{N_k})} \right) \stackrel{k \rightarrow \infty}{\Rightarrow} \mu_T$ in $\CP\!\left( \SC_{tem}^+ \right)$. We obtain for $\phi \in \SC_c^+$ arbitrary 
\eqn{
\lbeq{equ:proof_sup_of_initial} 
  \int e^{-2 \langle f,\phi \rangle} \mu_T(df)
  = \lim_{k \rightarrow \infty} \EE\!\left[ e^{-2 \langle u_T^{(\psi_{N_k})},\phi \rangle} \right] 
  = \lim_{k \rightarrow \infty} \EE\!\left[ e^{-2 \langle \psi_{N_k},u_T^{(\phi)} \rangle} \right] 
  = \PP\!\left( \langle u_T^{(\phi)},1 \rangle=0 \right). 
}
Here we used the duality from \eqref{equ:self-duality} in the second equality and the definition of $\psi_N$ together with dominated convergence in the last step. Use the definition of $\tau^{(\phi)}$ to rewrite the above as
\eqn{
\lbeq{equ:upper_limit_t}
  \int e^{-2 \langle f,\phi \rangle} \mu_T(df) 
  = \PP\!\left( \tau^{(\phi)} \leq T \right).
}
Uniqueness of the limit $\mu_T$ follows as the Laplace functional uniquely characterizes a measure.

It remains to show that \eqref{equ:upper_limit_t} holds for all $\phi \in \SC_{tem}^+$. Let $\phi_n \in \SC_c^+, \phi_n \uparrow \phi$. By monotonicity, $\PP\!\left( \tau^{(\phi)} \leq T \right) \leq \PP\!\left( \tau^{(\phi_n)} \leq T \right)$ for all $n \in \NN$. By the continuity of $f \mapsto \PP_f$ on $\SC_{tem}^+$, $\limsup_{n \rightarrow \infty} \PP_{\phi_n}( u_T \in \{0\} ) \leq \PP_\phi( u_T \in \{0\} ) = \PP(\tau^{(\phi)} \leq T)$ and $\lim_{n \rightarrow \infty} \PP\!\left( \tau^{(\phi_n)} \leq T \right) = \PP(\tau^{(\phi)} \leq T)$ follows. Use dominated convergence to pass to the limit in the left hand side of \eqref{equ:upper_limit_t}. 
\end{proof}
%
%
\begin{remark}
\label{RMK:upper_inv_time_change}
Let $T>0$ be arbitrarily fixed. By Theorem~\ref{THM:tribe}, $\SL\!\left( u_t^{(\mu_T)} \right) = \mu_{T+t}$ for all $t \geq 0$. Indeed, use self-duality to see that the Laplace functionals coincide: for all $\phi \in \SC_{tem}^+$,
\eqan{
  \EE_{\mu_T}\!\left[ e^{-2 \langle u_t , \phi \rangle} \right] 
  &= \int \EE_f\!\left[ e^{-2 \langle u_t , \phi \rangle} \right] \mu_T(df)
  = \int \EE_\phi\!\left[ e^{-2 \langle u_t , f \rangle} \right] \mu_T(df)
  = \EE_\phi\!\left[ \int e^{-2 \langle u_t , f \rangle} \mu_T(df) \right] \\
  &= \EE_\phi\!\left[ \PP_{u_t}( \tau \leq T) \right]
  = \PP_\phi( \tau \leq T+t)
  = \int e^{-2 \langle f,\phi \rangle} \mu_{T+t}(df) \nn
}
holds. 
\end{remark}
%
%
\begin{proposition} 
\label{PRO:u_star}
For $T> 0$ fixed, $\left\{ \mu_{T+t}, t \geq 0 \right\}$ is tight in $\CP(\SC_{tem}^+)$. In particular, $\mu_{T+t} \Rightarrow \mu$ for $t \rightarrow \infty$ in $\CP\!\left( \SC_{tem}^+ \right)$ where $\mu$ satisfies \eqref{equ:upper-invariant-Laplace} for all $g \in \SC_c^+$.
\end{proposition}
%
%
\begin{proof}
We obtain from \cite[Lemma~3.4]{T1996} and Remark~\ref{RMK:upper_inv_time_change} for all $\theta > 0, p \geq 2, |x-x'| \leq 1$ and $t \geq 0$,
\eqn{
  \int \left| f(x)-f(x') \right|^p \mu_{T+t}(df)
  = \int \EE_f\!\left[ \left| u_{T/2}(x)-u_{T/2}(x') \right|^p \right] \mu_{t+T/2}(df) 
  \leq C(\theta,p,T) |x-x'|^{p/2-1}. 
}
Similar reasoning, using \cite[Lemma~3.3]{T1996} yields $\int ( \langle f,e^{-|\cdot|} \rangle )^p \mu_{T+t}(df) \leq C(\theta,p,T)$. Reason as at the beginning of the proof of Proposition~\ref{PRO:tightness_T} to obtain (i)--(ii) of the tightness-conditions applied to $\{\mu_{T+t}, t \geq 0\}$. It follows that $\left\{ \mu_{T+t}, t \geq 0 \right\}$ is tight in $\CP\!\left( \SC_{tem}^+ \right)$ for $T>0$ arbitrarily fixed. The reminder of the proof is analogous to the proof of Proposition~\ref{PRO:tightness_T} once we observe that for any $\mu_{T_k} \Rightarrow \nu \in \CP( \SC_{tem}^+ )$ with $T_k \rightarrow \infty$ we have for $\phi \in \SC_c^+$ arbitrary,
\eqn{
\lbeq{equ:Laplace_X}
  \int e^{-2 \langle f,\phi \rangle} \nu(df) 
  = \lim_{k \rightarrow \infty} \int e^{-2 \langle f,\phi \rangle} \mu_{T_k}(df) 
  \stackrel{\eqref{equ:upper_limit_t}}{=} \lim_{k \rightarrow \infty} \PP\!\left( \tau^{(\phi)} \leq T_k \right) 
  = \PP\!\left( \tau^{(\phi)} < \infty \right), 
}
the last by monotone convergence. This concludes the proof.
\end{proof}
%
%

We obtain in particular the following from the above proofs.
%
%
\begin{remark}
Use self-duality to get for all $T>0$ and $\phi \in \SC_{tem}^+$,
\eqn{
  \int e^{-2 \langle f,\phi \rangle} \mu_T(df) = \PP\!\left( \tau^{(\phi)} \leq T \right) = \sup_{\psi \in \SC_{tem}^+} \EE\!\left[ e^{-2<u_T^{(\psi)},\phi>} \right] 
}
and for all $\phi \in \SC_c^+$,
\eqn{
  \int e^{-2 \langle f,\phi \rangle} \mu(df) = \PP\!\left( \tau^{(\phi)} < \infty \right) = \lim_{t \rightarrow \infty} \sup_{\psi \in \SC_{tem}^+} \EE\!\left[ e^{-2<u_{T+t}^{(\psi)},\phi>} \right]. 
}
\end{remark}
%
%
\begin{corollary}[Upper measure] 
\label{COR:coupling_dominated_by_upper}
Let $u_0 \in \SC_{tem}^+$ and $T>0$ be arbitrarily fixed. Then there exists a coupling of a solution $u^{(u_0)}$ to \eqref{equ:SPDE} and a random continuous process $(u^*_{T+t})_{t \geq 0}$ with values in $\SC_{tem}^+$ such that 
\eqn{
  u_{T+t}^{(u_0)}(x) \leq u_{T+t}^*(x) \quad \mbox{ for all } t \geq 0, x \in \RR \mbox{ almost surely}
}
and $\SL((u_{T+t}^*)_{t \geq 0})=\PP_{\mu_T}$. 
\end{corollary}
%
%
\begin{proof}
Choose a sequence $(\psi_N)_{N \in \NN}$ as in Proposition~\ref{PRO:tightness_T} such that $\psi_1 \geq u_0$. By reasoning as in \eqref{equ:coupling-B}--\eqref{equ:coupling-IT} one can construct a monotonically increasing sequence of solutions $u^{(0)}, u^{(N)}, N \in \NN$ to \eqref{equ:SPDE} with initial conditions $u_0$ respectively $\psi_N, N \in \NN$ on a common probability space. For $T>0$ fixed, let $u^*_{T+t}(x) \equiv \uparrow \lim_{N \rightarrow \infty} u^{(N)}_{T+t}(x), t \geq 0, x \in \RR$. Then $u^*_{T+\cdot}$ has law $\PP_{\mu_T}$ by Proposition~\ref{PRO:tightness_T} and Theorem~\ref{THM:tribe} 
\end{proof}
%
%
\begin{remark}[Remark on notation]
Let $G$ be a function from $\SC_{tem}^+$ to $\RR$. In what follows we often use the notations $\SL(u_t^*)$ and $\EE[G(u_t^*)]$ in place of $\mu_t$ and $\int G(f) \mu_t(df)$ to remind the reader of the dominating property of $\mu_t$. 
\end{remark}
%
%
\begin{remark}[Left upper measure and right upper measure] 
\label{RMK:on_u_star_l} 
\hfill
\begin{enumerate}
\item[(i)]
By analoguous reasoning to the above one can use non-decreasing sequences $\zeta_N \in \SC_{tem}^+$ such that $\zeta_N(x) \uparrow \infty$ for $x<0$ and $\zeta_N(x) = 0$ for $x \geq 0$ to prove that for $T>0$ arbitrarily fixed there exists a \textbf{left upper measure} $\upsilon_T \in \CP(\SC_{tem}^+)$ uniquely characterized by its Laplace functional
\eqn{
  \int e^{-2 \langle f,g \rangle} \upsilon_T(df) = \PP\!\left( \big\langle \1_{(-\infty,0)}(\cdot) , u_T^{(g)} \big\rangle = 0 \right), \quad \mbox{ for } g \in \SC_{tem}^+
}
and such that $\SL\!\left( u_T^{(\zeta_N)} \right) \Rightarrow \upsilon_T$ in $\CP(\SC_{tem}^+)$. Moreover, $\SL(u_t^{(\upsilon_T)}) = \upsilon_{T+t}$.
\item[(ii)]
For $u_0 \in \SC_{tem}^+$ with $R_0(u_0) \leq 0$ and $T>0$ arbitrarily fixed, using $\zeta_1 \geq u_0$ in the construction, one obtains analoguously the existence of a coupling such that 
\eqn{
  u_{T+t}^{(u_0)}(x) \leq u_{T+t}^{*,l}(x) \quad \mbox{ for all } x \in \RR, t \geq 0 \mbox{ almost surely,}
}
where $\SL((u_{T+t}^{*,l})_{t \geq 0})=\PP_{\upsilon_T}$ holds. Note in particular that such a coupling yields
\eqn{
\lbeq{equ:dom-left-upper}
  R_0(u_{T+t}^{(u_0)}) \leq R_0(u_{T+t}^{*,l}) \mbox{ for all } t \geq 0 \mbox{ almost surely.} 
}
One can further show that there exists a coupling such that for $T>0$ fixed, 
\eqn{
  u_{T+t}^{*,l}(x) \leq u_{T+t}^*(x), \quad \mbox{ for all } t \geq 0, x \in \RR \mbox{ almost surely}.
}
\item[(iii)]
The construction of a \textbf{right upper measure} $\kappa_T$ by means of $\xi_N(x) \equiv \zeta_N(-x)$ and corresponding properties follow analoguously. We will use the notation $\SL(u_t^{*,r}) \equiv \kappa_t$ in what follows. 
\end{enumerate}
\end{remark}
%
%
\begin{corollary}
For $\phi \in \SC_{tem}^+$ and $T>0$ arbitrary,
\eqn{ 
\lbeq{equ:bound_expectation_u_star_t}
  \EE[ \langle u_T^* , \phi \rangle ] \leq \theta \frac{\langle \phi,1\rangle}{1-e^{-\theta T}},
}
\eqn{
\lbeq{equ:bound_expectation_u_star_t-l}
  \EE\!\left[ \langle u_T^{*,l} , \phi \rangle \right]
  = \lim_{\lambda \rightarrow 0^+} \frac{ \PP_{\lambda \phi}\!\left( u_T|_{x<0} \not\equiv 0 \right) }{2\lambda}
}
and
\eqn{ 
\lbeq{equ:bound_expectation_u_star_infty}
  \int \langle f , \phi \rangle \mu(df) \leq \theta \langle \phi,1\rangle.
}
\end{corollary}
%
%
\begin{proof}
Use \eqref{equ:Laplace-u-T} and coupling with a superprocess (see Remark~\ref{RMK:monotonicity-domination_super}(ii)) and \eqref{equ:bound_die_out_t} to see that
\eqan{
  \EE\!\left[ \langle u_T^*,\phi \rangle \right] 
  &= (-1/2) \frac{d}{d\lambda} \Big|_{\lambda=0^+} \EE\!\left[ e^{-2\lambda \langle u_T^*,\phi \rangle} \right] 
  = (-1/2) \frac{d}{d\lambda} \Big|_{\lambda=0^+} \PP\!\left( \tau^{(\lambda \phi)} \leq T \right) \\
  &= (-1/2) \lim_{\lambda \downarrow 0^+} \frac{\PP\!\left( \tau^{(\lambda \phi)} \leq T \right) - 1}{\lambda} 
  \leq (-1/2) \lim_{\lambda \downarrow 0^+} \frac{\overline{\PP}\!\left( \bar{\tau}^{(\lambda \phi)} \leq T \right) - 1}{\lambda} \nn\\
  &= (-1/2) \lim_{\lambda \downarrow 0^+} \frac{e^{-2\theta \frac{\langle \lambda \phi,1 \rangle}{1-e^{-\theta T}}} - 1}{\lambda} 
  = (-1/2) \frac{d}{d\lambda} \Big|_{\lambda=0^+} e^{-2\theta\lambda \frac{\langle \phi,1 \rangle}{1-e^{-\theta T}}} 
  = \theta \frac{\langle \phi,1\rangle}{1-e^{-\theta T}}. \nn
}

The second claim follows by analoguous reasoning. For the third claim, let $M>0$ fixed, $\phi_n \in \SC_c^+, n \in \NN$ such that $\phi_n \uparrow \phi$. By Proposition~\ref{PRO:u_star} and \eqref{equ:bound_expectation_u_star_t},
\eqn{
\lbeq{equ:bound_expectation_u_star-calc}
  \int \left( \langle f , \phi_n \rangle \wedge M \right) \mu(df)
  = \lim_{T \rightarrow \infty} \int \left( \langle f , \phi_n \rangle \wedge M \right) \mu_T(df)
  \leq \lim_{T \rightarrow \infty} \theta \frac{\langle \phi_n,1\rangle}{1-e^{-\theta T}}
  = \theta \langle \phi_n , 1 \rangle.
}
Use monotone convergence first for $M \rightarrow \infty$, then for $\phi_n \uparrow \phi$ to establish the claim.
\end{proof}
%
%
\section{Blow up of the overall mass and support}
\label{SEC:mass_and_support}
%
%
\begin{proposition} 
\label{PRO:overall_mass_and_support}
Let $\theta>\theta_c$ and $u_0 \in \SC_{tem}^+ \backslash \{0\}$. Then the overall mass process satisfies in the limit
\eqn{
  \PP_{u_0}\!\left( \lim_{t \rightarrow \infty} \langle u_t,1 \rangle = \infty \left. \right| \tau = \infty \right) = 1. 
}
Moreover, if $u_0 \in \SC_c^+$, 
\eqn{
  \PP_{u_0}\!\left( \lim_{t \rightarrow \infty} \left( R_0(t)-L_0(t) \right) = \infty \left. \right| \tau = \infty \right) = 1. 
}
\end{proposition}
%
%
\begin{proof}
If the process has infinite mass at all times, the first claim is trivial. Thus, without loss of generality let $\langle u_0 , 1 \rangle < \infty$ in what follows (otherwise condition on the first time the process has finite mass). Consider $M \in \NN$ arbitrary such that $<u_0,1> < M$. Set $T_0 \equiv 0$ and
\eqn{
  T_m \equiv \inf\!\left\{ t \geq T_{m-1} + 1: \langle u_t,1 \rangle \leq M \right\}, \ m \in \NN, 
}
where we set $\inf \emptyset = \infty$. Next observe that 
\eqn{
  \PP_{u_0}\!\left( T_m < \tau \right) = \PP_{u_0}\!\left( T_m < \tau , T_m < \infty \right) 
  = \EE_{u_0}\!\left[ \EE_{u_0}\!\left[ \1\!\left( T_m < \tau , T_m < \infty \right) \left. \right| \SF_{T_{m-1}} \right] \prod_{i=1}^{m-1} \1\!\left( T_i < \tau , T_i < \infty \right) \right]. 
}
Using the strong Markov property of the process $u=u^{(u_0)}$ and the definition of $T_1=T_1^{(u_0)}$ we get as a further upper bound
\eqan{
  & \EE_{u_0}\!\left[ \PP_{u_{T_{m-1}}^{(u_0)}}\!\left( T_1^{\left( u_{T_{m-1}}^{(u_0)} \right)} < \tau^{\left( u_{T_{m-1}}^{(u_0)} \right)} , T_1^{\left( u_{T_{m-1}}^{(u_0)} \right)} < \infty \right) \prod_{i=1}^{m-1} \1\!\left( T_i^{(u_0)} < \tau^{(u_0)} , T_i^{(u_0)} < \infty \right) \right] \\
  & \leq \EE_{u_0}\!\left[ \PP_{u_{T_{m-1}}^{(u_0)}}\!\left(1 < \tau^{\left( u_{T_{m-1}}^{(u_0)} \right)} \right) \prod_{i=1}^{m-1} \1\!\left( T_i^{(u_0)} < \tau^{(u_0)} , T_i^{(u_0)} < \infty \right) \right]. \nn
}
Let $p_M \equiv \sup \limits_{u_0: \langle u_0,1 \rangle \leq M} \PP_{u_0}\!\left( 1 < \tau \right)$, then we obtain by iterating the above
\eqn{
  \PP_{u_0}\!\left( T_m < \tau , T_m < \infty \right) \leq p_M \EE_{u_0}\!\left[ \prod_{i=1}^{m-1} \1\!\left( T_i < \tau , T_i < \infty \right) \right] \leq p_M^m. 
}
To obtain an upper bound on $p_M$ we will dominate $u$ by a superprocess $\bar{u}$. Now \eqref{equ:bound_die_out_t} yields 
\eqn{
   \sup \limits_{u_0: \langle u_0,1 \rangle \leq M} \PP_{u_0}\!\left( 1 < \tau \right) \leq 1-\exp\!\left( \frac{ -2 \theta M }{ 1-e^{-\theta } } \right). 
}
Hence $p_M < 1$ and
\eqn{
  \PP_{u_0}\!\left( T_m < \tau \right) 
  = \PP_{u_0}\!\left( T_m < \tau , T_m < \infty \right) \leq p_M^m 
  \rightarrow 0 \mbox{ for } m \rightarrow \infty 
}
for all $u_0 \in \SC_{tem}^+$ satisfying $0< \langle u_0,1 \rangle < M$. By continuity of measures, this finally gives with $T_m=T_m^{(M)}$ as above,
\eqan{
  \PP_{u_0}\!\left( \liminf_{t \rightarrow \infty} \langle u_t , 1 \rangle < \infty , \tau = \infty \right) 
  & = \lim_{M \rightarrow \infty} \PP_{u_0}\!\left( \liminf_{t \rightarrow \infty} \langle u_t , 1 \rangle < M , \tau = \infty \right) \nn\\
  & \leq \lim_{M \rightarrow \infty} \PP_{u_0}\!\left( T_m^{(M)} < \tau \ \forall m \in \NN , \tau = \infty \right) \nn\\
  & \leq \lim_{M \rightarrow \infty} \lim_{m \rightarrow \infty} \PP_{u_0}\!\left( T_m^{(M)} < \tau \right) \nn\\
  & \leq \lim_{M \rightarrow \infty} \lim_{m \rightarrow \infty} p_M^m \nn\\
  & = 0, \nn
}
which proves the first part of the claim. The second part can be shown by similar reasoning, using \eqref{equ:bound_die_out_interval_t} in place of \eqref{equ:bound_die_out_t}.
\end{proof}
%
%
\section{Estimates on the right wavefront marker}
\label{SEC:estimates_wfm}
%
%
One of the first results is the almost sure finiteness of the positive part of the right wavefront marker of $u^{*,l}_T$ (recall Remark~\ref{RMK:on_u_star_l}(ii)) for times $T>0$. Indeed, the expectation turns out to be linearly bounded in time $T$. 

The strategy to the proof of such a result is as follows. For fixed $T>0$, consider $\big( u_{T/2+t}^{*,l} \big)_{t \geq 0}$. For $t=0$, $u_{T/2}^{*,l}$ might or might not satisfy $R_0(T/2)=R_0(u_{T/2}^{*,l}) < \infty$, but one shows that with high probability, to the right of some large enough $R>0$ it only has finite mass. For the remaining time $T/2$, kill off enough parts to the right of $R$ to regain $R_0(T) < \infty$ for this part of the solution. For the part to the left of $R$, use the compact support property. Note that this strategy works for \eqref{equ:SPDE} as local patches of high mass immediately get ``beaten down`` due to the additional drift of $-u^2$. Therefore the speed of the right front of the support of a solution over a fixed time-interval is determined rather by the shape of the right wavefront than by the overall shape of the solution. 
%
%
\begin{remark}[Notation involving non-continuous initial conditions]
\label{RMK:non-cont}
In the following three lemmas and the subsequent proof of Proposition~\ref{PRO:bound_on_u_star_l_t} initial conditions to \eqref{equ:SPDE} appear that involve indicator functions and thus are not continuous. This notation should be understood as explained below and is only used as an abbreviation to facilitate following the main idea of the proofs.

For $\nu \in \CP(\SC_{tem}^+), t \geq 0, A \subset \RR, F: \SC_{tem}^+ \rightarrow \RR^+$ measurable, denote
\eqn{
  \int \EE_{f \1_A(\cdot)}\!\left[ F(u_t) \right] \nu(df) \equiv \inf_{\phi \in \SC_{tem}^+, \phi \geq \1_A} \int \EE_{f \phi}\!\left[ F(u_t) \right] \nu(df).
}
The proofs should then be executed for appropriate $\phi$ fixed in place of $\1_A(\cdot)$. Only in the conclusion of the proofs the infimum over $\phi \in \SC_{tem}^+, \phi \geq \1_A$ is taken to conclude the claims.
\end{remark} 
%
%
\begin{lemma}
\label{LEM:bd-part-of-exp}
For $\theta > 0$ and $0<T \leq 1$,
\eqn{
  \EE\!\left[ \EE_{u_{T/2}^{*,l} \1_{(-\infty,T^{1/4}]}(\cdot)}\!\left[ 0 \vee R_0(T/2) \right] \right] 
  \leq C(\theta) T^{1/4}.
}
\end{lemma}
%
%
\begin{proof}
Recall the notation of Remark~\ref{RMK:on_u_star_l}. We first bound
%
\eqan{
  & \EE\!\left[ \EE_{u_{T/2}^{*,l} \1_{(-\infty,T^{1/4}]}(\cdot)}\!\left[ 0 \vee R_0(T/2) \right] \right] 
  \leq T^{1/4} + \int \EE_{f(\cdot + T^{1/4}) \1_{(-\infty,0]}(\cdot)}\!\left[ 0 \vee R_0(T/2) \right] \upsilon_{T/2}(df) \\
  &\leq 2 T^{1/4} + \int \int_{T^{1/4}}^\infty \PP_{f(\cdot+T^{1/4}) \1_{(-\infty,0]}(\cdot)}\!\left( \sup_{0 \leq s \leq T/2} R_0(s) > R \right) dR \ \upsilon_{T/2}(df). \nn
}
Lemma~\ref{LEM:tribe_lemma_2_1} from the Appendix yields for $R>q>0$ arbitrary,
\eqn{
  \PP_{f(\cdot+T^{1/4}) \1_{(-\infty,0]}(\cdot)}\!\left( \sup_{0 \leq s \leq T/2} R_0(s) > R \right) 
  \leq C \left( \theta + (R-q)^{-2} \right) e^{\theta T/2} \int_{-\infty}^{T^{1/4}} \exp\!\left( -\frac{(q-x+T^{1/4})^2}{2T} \right) f(x) dx. 
}
Abbreviate
\eqn{
  \Phi^{[R,q,T]}(x) \equiv \1_{(-\infty,T^{1/4} ]}(x) \exp\!\left( - \frac{(q-x+T^{1/4})^2}{2T} \right),
}
then by \eqref{equ:bound_expectation_u_star_t-l}, 
\eqn{
  \EE\!\left[ \EE_{u_{T/2}^{*,l} \1_{(-\infty,T^{1/4}]}(\cdot)}\!\left[ 0 \vee R_0(T/2) \right] \right] 
  \leq 2 T^{1/4} + C \int_{T^{1/4}}^\infty \left( \theta + (R-q)^{-2} \right) e^{\theta T/2} \lim_{\lambda \rightarrow 0^+} \frac{ \PP_{\lambda \Phi^{[R,q,T]}}\!\left( u_{T/2}|_{x<0} \not\equiv 0 \right) }{2\lambda} dR.
}
By the crude bounds $\PP_{\Phi^{[R,q,T]}}\!\left( u_{T/2}|_{x<0} \not\equiv 0 \right) \leq \PP_{\Phi^{[R,q,T]}}\!\left( \tau>T/2 \right)$, $1-e^{-x} \leq x$ for $x \geq 0$ and \eqref{equ:bound_die_out_t},
\eqn{
  \EE\!\left[ \EE_{u_{T/2}^{*,l} \1_{(-\infty,T^{1/4}]}(\cdot)}\!\left[ 0 \vee R_0(T/2) \right] \right] 
  \leq 2 T^{1/4} + C \int_{T^{1/4}}^\infty \left( \theta + (R-q)^{-2} \right) e^{\theta T/2} \theta \frac{\langle \Phi^{[R,q,T]} , 1 \rangle}{1-e^{-\theta T/2}} dR.
}
Choose $q=R/2$ and recall that $0<T\leq 1$ to get, using $1-\exp(-x) \geq C(c) x$ for $0 \leq x \leq c$,
\eqan{
  \EE\!\left[ \EE_{u_{T/2}^{*,l} \1_{(-\infty,T^{1/4}]}(\cdot)}\!\left[ 0 \vee R_0(T/2) \right] \right] 
  &\leq 2 T^{1/4} + C(\theta) (1+T^{-1/2}) T^{-1} \int_{T^{1/4}}^\infty \int_{-\infty}^{T^{1/4}}  e^{-\frac{(R/2-x+T^{1/4})^2}{2T}} dx dR \\
  &\leq 2 T^{1/4} + C(\theta) (1+T^{-1/2}) T^{-1} T^{1/2} \int_{T^{1/4}}^\infty e^{-\frac{(R/2)^2}{4T}} dR \nn\\
  &\leq 2 T^{1/4} + C(\theta) (1+T^{-1/2}) e^{-\frac{1}{32 T^{1/2}}}, \nn   
}
which concludes the proof.
\end{proof}
%
%
\begin{lemma}
\label{LEM:bd-small-t-small-mass-survives}
For $\theta > 0$, $0<T \leq 1, R>0$ and $n \in \NN$ arbitrary,
\eqn{
  \EE\!\left[ \PP_{u_{T/2}^{*,l} \1_{(nR,(n+1)R]}(\cdot)}( \tau>T/2 ) \right]
  \leq C(\theta) T^{-1/2} \left( \theta + (R/2)^{-2} \right) e^{-\frac{((n-1/2)R)^2}{4T}}.
}
\end{lemma}
%
%
\begin{proof}
Domination by a superprocess and \eqref{equ:bound_die_out_t} yield
\eqn{
  \EE\!\left[ \PP_{u_{T/2}^{*,l} \1_{(nR,(n+1)R]}(\cdot)}( \tau>T/2 ) \right]
  \leq C(\theta) T^{-1} \EE\!\left[ \big\langle u_{T/2}^{*,l} \1_{(nR,(n+1)R]}(\cdot) , 1 \big\rangle \right]. 
}
To further bound the right hand side, use \eqref{equ:bound_expectation_u_star_t-l} and Lemma~\ref{LEM:tribe_lemma_2_1} to obtain 
\eqan{
\lbeq{equ:bound-u-star-far-mass}
  \EE\!\left[ \big\langle u_{T/2}^{*,l} , \1_{(nR,\infty)}(\cdot) \big\rangle \right]
  &= \lim_{\lambda \rightarrow 0^+} \frac{ \PP_{\lambda \1_{(-\infty,0)}(\cdot)}\!\left( u_{T/2}|_{(nR,\infty)} \not\equiv 0 \right) }{2\lambda} \\
  &\leq C \left( \theta + (nR-q)^{-2} \right) e^{\theta T/2} \int_{-\infty}^0 \exp\!\left( - \frac{(q-x)^2}{2T} \right) dx \nn\\
  &\leq C(\theta) \left( \theta + (nR-q)^{-2} \right) T^{1/2} e^{-\frac{q^2}{4T}} \nn
}
with $q=q(n,R), nR>q>0$ arbitrarily fixed. The choice $q=(n-1/2)R$ gives the claim.
\end{proof}
%
%
\begin{lemma}
\label{LEM:bd-small-t-small-mass-moves}
For $\theta > 0$, $0<T \leq 1, R>0$ and $n \in \NN$ arbitrary,
\eqn{
\lbeq{equ:bd-small-t-small-mass-moves}
  \EE\!\left[ \EE_{u_{T/2}^{*,l} \1_{(nR,(n+1)R]}(\cdot)}\!\left[ (0 \vee R_0(T/2))^2 \right] \right] 
  \leq 4((n+1)R)^2 + C(\theta) T^{-1/2} e^{-\frac{((n+1)R)^2}{32 T}}. 
}
\end{lemma}
%
%
\begin{proof}

To get a first upper bound, apply Lemma~\ref{LEM:tribe_lemma_2_1} to $f \in \SC_{tem}^+$ with $R_0(f) \leq (n+1)R$, to obtain
%
\eqan{
  & \EE_f\!\left[ \sup_{0 \leq s \leq T/2} (0 \vee R_0(s))^2 \right]
  \leq ((n+1)R+Q)^2 + \int_{(n+1)R+Q}^\infty 2\tilde{R} \PP_f\!\left( R_0(T/2) > \tilde{R} \right) d\tilde{R} \\
  &\leq ((n+1)R+Q)^2 + C(\theta) \int_{(n+1)R+Q}^\infty \tilde{R} \left( \theta + (\tilde{R}-\tilde{q})^{-2} \right) \int \exp\!\left( - \frac{(\tilde{q}-x)^2}{2T} \right) f(x) dx d\tilde{R} \nn
}
with $Q>0$ and $\tilde{q}=\tilde{q}(\tilde{R}), \tilde{R}>\tilde{q}>(n+1)R$ arbitrary. Integration against $\upsilon_{T/2}(df)$ yields as an upper bound to the right hand side in \eqref{equ:bd-small-t-small-mass-moves},
\eqn{
  ((n+1)R+Q)^2 + C(\theta) \int_{(n+1)R+Q}^\infty \tilde{R} \left( \theta + (\tilde{R}-\tilde{q})^{-2} \right) \EE\!\left[ \big\langle \Psi^{[n,R,\tilde{q},T]} , u_{T/2}^{*,l} \big\rangle \right] d\tilde{R}
}
with 
\eqn{
  \Psi^{[n,R,\tilde{q},T]}(x) \equiv \1_{(nR,(n+1)R]}(x) \exp\!\left( - \frac{(\tilde{q}-x)^2}{2T} \right).
}
Dominate $u^{*,l}$ by $u^*$ and then apply \eqref{equ:bound_expectation_u_star_t} to get as a further upper bound
\eqn{
  ((n+1)R+Q)^2 + C(\theta) T^{-1} \int_{(n+1)R+Q}^\infty \tilde{R} \left( \theta + (\tilde{R}-\tilde{q})^{-2} \right) \langle \Psi^{[n,R,\tilde{q},T]} , 1 \rangle d\tilde{R}.
}
With the substitutions $\bar{R}+(n+1)R=\tilde{R}$, $\bar{q}+(n+1)R=\tilde{q}$ and $\bar{x}=(n+1)R+x$ this reads
\eqn{
  ((n+1)R+Q)^2 + C(\theta) T^{-1} \int_Q^\infty (\bar{R}+(n+1)R) \left( \theta + (\bar{R}-\bar{q})^{-2} \right) \big\langle \Psi^{[-1,R,\bar{q},T]} , 1 \big\rangle d\bar{R}
}
for $\bar{R}>\bar{q}>0$. Choose $Q=(n+1)R$ and $\bar{q}=\bar{R}/2$ to conclude that this in turn can be bounded from above by (recall that we assume $0<T \leq 1$)
\eqan{
  & 4((n+1)R)^2 + C(\theta) T^{-1} \int_{(n+1)R}^\infty \bar{R} \left( \theta + (\bar{R}/2)^{-2} \right) e^{-\frac{\bar{R}^2}{16 T}} \int_{-R}^0 e^{-\frac{(\bar{R}/2-\bar{x})^2}{4T}} d\bar{x} d\bar{R} \\
  & \leq 4((n+1)R)^2 + C(\theta) T^{-1/2} e^{-\frac{((n+1)R)^2}{32 T}} \int_{(n+1)R}^\infty \frac{\bar{R}}{\sqrt{T}} \left( \theta + \left( \frac{\bar{R}}{2\sqrt{T}} \right)^{-2} \right) e^{-\frac{\bar{R}^2}{32 T}} \frac{1}{\sqrt{T}} d\bar{R} \nn\\
  & \leq 4((n+1)R)^2 + C(\theta) T^{-1/2} e^{-\frac{((n+1)R)^2}{32 T}}, \nn
}
which completes the claim.
\end{proof}
%
%
\begin{proposition} 
\label{PRO:bound_on_u_star_l_t}
For $\theta>0$ and $0<T \leq 1$,
\eqn{
  \EE\!\left[ 0 \vee R_0\!\left( u_T^{*,l} \right) \right] \leq C(\theta) T^{1/4}.
}
\end{proposition}
%
%
\begin{proof}
Let $R \equiv T^{1/4}$. Apply the  monotonicity property from Remark~\ref{RMK:monotonicity-domination_super}(i) to the following countable sum of initial conditions to get
\eqan{
  \EE\!\left[ 0 \vee R_0\!\left( u_T^{*,l} \right) \right] \lbeq{equ:infinite_sum} 
  & = \EE\!\left[ \EE_{u_{T/2}^{*,l}}\!\left[ 0 \vee R_0(T/2) \right] \right] \\
  & \leq \EE\!\left[ \EE_{u_{T/2}^{*,l} \1_{(-\infty,R]}(\cdot)}\!\left[ 0 \vee R_0(T/2) \right] \right] + \sum_{n \geq 1} \EE\!\left[ \EE_{u_{T/2}^{*,l} \1_{(nR,(n+1)R]}(\cdot)}\!\left[ 0 \vee R_0(T/2) \right] \right]. \nn
}
The first term can be bounded by $C(\theta) T^{1/4}$ by Lemma~\ref{LEM:bd-part-of-exp}. To bound the summands of the second term, apply Cauchy-Schwarz' inequality twice: For fixed $n \in \NN$ and with the abbreviation $f = f(\omega) \equiv u_{T/2}^{*,l}(\omega) \cdot \1_{(nR,(n+1)R]}(\cdot)$ we have 
\eqan{
\lbeq{equ:both_factors} 
  \EE\!\left[ \EE_f\!\left[ 0 \vee R_0(T/2) \right] \right] 
  & = \EE\!\left[ \EE_f\!\left[ (0 \vee R_0(T/2)) \1_{\{\tau>T/2\}} \right] \right] 
  \leq \EE\!\left[ \left( \EE_f\!\left[ (0 \vee R_0(T/2))^2 \right] \PP_f( \tau>T/2 ) \right)^{1/2} \right] \\
  & \leq \left( \EE\!\left[ \EE_f\!\left[ (0 \vee R_0(T/2))^2 \right] \right] \EE\!\left[ \PP_f( \tau>T/2 ) \right] \right)^{1/2}. \nn
}

To bound the first factor use Lemma~\ref{LEM:bd-small-t-small-mass-moves}, to bound the second factor use Lemma~\ref{LEM:bd-small-t-small-mass-survives}. Collecting terms, we obtain
\eqan{
  & \EE\!\left[ 0 \vee R_0\!\left( u_T^{*,l} \right) \right] \\
  &\leq C(\theta) T^{1/4} + \sum_{n \geq 1} \left( \left\{ 4((n+1)R)^2 + C(\theta) T^{-1/2} e^{-\frac{((n+1)R)^2}{32 T}} \right\} C(\theta) T^{-1/2} \left( \theta + (R/2)^{-2} \right) e^{-\frac{((n-1/2)R)^2}{4T}} \right)^{1/2}. \nn
}
Recall the choice $R=T^{1/4}$ to conclude
\eqn{
  \EE\!\left[ 0 \vee R_0\!\left( u_T^{*,l} \right) \right] 
  \leq C(\theta) T^{1/4} + C(\theta) e^{-\frac{1}{64 \sqrt{T}}} \sum_{n \geq 1} (n+1)^2 e^{-\frac{(n-1/2)^2}{16 \sqrt{T}}}
  \leq C(\theta) T^{1/4}
}
as claimed. 
\end{proof}
%
%
\begin{lemma} 
\label{LEM:bound_on_u_star_l_T}
For all $\theta>0, T \geq 1$,
\eqn{
  \EE\!\left[ 0 \vee R_0\!\left( u_T^{*,l} \right) \right] \leq C(\theta) T. 
}
\end{lemma}
%
%
\begin{proof}
We first note that for all $n \in \NN$, $\EE\!\left[ 0 \vee R_0(u_n^{*,l}) \right] < \infty$. Indeed, use induction: The claim follows for $n=1$ directly from Proposition~\ref{PRO:bound_on_u_star_l_t}. Suppose the claim holds for $n$ fixed. Remark~\ref{RMK:on_u_star_l}(ii) yields a coupling such that $u_1^{(u_n^{*,l})}(\cdot + R_0(u_n^{*,l})) \leq v$ holds a.s. with $\SL(v) \sim \upsilon_1$. As a result, 
\eqan{
  \EE\!\left[ 0 \vee R_0(u_{n+1}^{*,l}) \right] 
  &\leq \EE\!\left[ 0 \vee R_0(u_n^{*,l}) \right] + \EE\!\left[ 0 \vee R_0\!\left( u_1^{(u_n^{*,l})}(\cdot + R_0(u_n^{*,l})) \right) \right] \\
  &\leq \EE\!\left[ 0 \vee R_0(u_n^{*,l}) \right] + \EE\!\left[ 0 \vee R_0(v) \right] < \infty. \nn
}

Use Remark~\ref{RMK:on_u_star_l}(ii) and Proposition~\ref{PRO:bound_on_u_star_l_t} again (let $R_0(t)-R_0(s) \equiv 0$ for $0 \leq s \leq t$ and $u(s) \equiv 0$, then the well-definiteness of the differences in wave-markers follows by the above) to obtain
\eqan{
  \EE\!\left[ 0 \vee R_0\!\left( u_T^{*,l} \right) \right] 
  &\leq \EE\!\left[ 0 \vee R_0\!\left( u_1^{*,l} \right) \right] \\
  &\quad + \sum_{n=1}^{\left\lfloor T \right\rfloor - 1} \EE\!\left[ \EE_{u_n^{*,l}}\!\left[ 0 \vee (R_0(1)-R_0(0)) \right] \right] + \EE\!\left[ \EE_{u_{\left\lfloor T \right\rfloor}^{*,l}}\!\left[ 0 \vee (R_0(T- \left\lfloor T \right\rfloor )-R_0(0)) \right] \right] \nn\\
  &\leq \left\lfloor T \right\rfloor \EE\!\left[ 0 \vee R_0\!\left( u_1^{*,l} \right) \right] + \EE\!\left[ 0 \vee R_0\!\left( u_{T - \left\lfloor T \right\rfloor}^{*,l} \right) \right]
  \leq \left( \left\lfloor T \right\rfloor + 1 \right) C(\theta) \nn
}
and the claim follows for all $T \geq 1$ after an appropriate change of constant.
\end{proof}
%
%
\begin{corollary} 
\label{COR:exp_pos_wavefront_bd}
For all $\theta>0$ and $u_0 \in \SC_{tem}^+$ with $R_0(u_0) \leq 0$ there exists $C(\theta)<\infty$ independent of $u_0$ such that
\eqn{
  \EE_{u_0}[0 \vee R_0(T)] \leq C(\theta) \left( T \vee T^{1/4} \right) 
}
for all $T \geq 0$.
\end{corollary}
%
%
\begin{proof}
The claim follows by Proposition~\ref{PRO:bound_on_u_star_l_t}, Lemma~\ref{LEM:bound_on_u_star_l_T} and monotonicity.
\end{proof}
%
%
Recall the definition of $L_0(f) = \inf\{ x \in \RR: f(x)>0 \}$.
%
%
\begin{lemma} 
\label{LEM:exp_wavefront_bd}
Suppose $\theta>0$ and $g_0 \in \SC_c^+ \backslash \{0\}$ arbitrarily fixed. There exists $C(\theta)<\infty$ independent of $g_0$ such that for all $T \geq 0$,
\eqn{
  \EE_{g_0}[|R_0(T)| \1_{\{\tau>T\}}] 
  \leq |R_0(g_0)| + |L_0(g_0)| + C(\theta) \left( T \vee T^{1/4} \right) 
}
holds.
\end{lemma}
%
%
\begin{proof}
Corollary~\ref{COR:exp_pos_wavefront_bd} yields
\eqan{
  \EE_{g_0}[0 \vee R_0(T)] 
  &\leq |R_0(g_0)| + \EE_{g_0}[0 \vee (R_0(T)-R_0(0))] \\
  &= |R_0(g_0)| + \EE_{g_0(\cdot+R_0(0))}[0 \vee R_0(T)]
  \leq |R_0(g_0)| + C(\theta) \left( T \vee T^{1/4} \right). \nn
}

Further consider $v_t(x) \equiv u_t(-x)$ for all $t \geq 0$, where $u$ is a solution to \eqref{equ:SPDE} starting in $g_0$. As $g_0$ has compact support, $v_0 \in \SC_c^+$ and $v$ is a solution to \eqref{equ:SPDE} starting in $v_0$. Note in particular that
\eqn{
  R_0(v_T)=-L_0(u_T) \mbox{ on } \{ T < \tau \}
}
and $\EE_{v_0}[0 \vee R_0(T)] \leq |R_0(v_0)| + C(\theta) (T \vee T^{1/4})$. Then
\eqan{
  \EE_{g_0}[|R_0(T)| \1_{\{\tau>T\}}] 
  &= \EE_{g_0}[0 \vee R_0(T)] + \EE_{g_0}[-(0 \wedge R_0(T)) \1_{\{\tau>T\}}] \\
  &\leq \EE_{g_0}[0 \vee R_0(T)] + \EE_{g_0}[-(0 \wedge L_0(T)) \1_{\{\tau>T\}}] 
  = \EE_{g_0}[0 \vee R_0(T)] + \EE_{v_0}[0 \vee R_0(T)] \nn
}
and the claim follows.
\end{proof}
%
%
%
\begin{center}
\framebox{For the remainder of the article assume $\theta>\theta_c$, unless otherwise indicated.}
\end{center}
%
%
%

We now prove a result in the spirit of \cite[Lemma~3.6]{T1996}. Recall the definition of $\PP_{\nu_T}$ with $\nu_T=\nu_T(g_0)$ from \eqref{equ:def_nu_T}. 
%
%
\begin{lemma} 
\label{LEM:3_6}
If $\theta>\theta_c, t>0$ and $g_0 \in \SC_c^+ \backslash \{0\}$, then there exists $C(g_0,\theta,t)$ such that for all $a > 0, 0 \leq s \leq t$ and $T \geq 1$,
\eqn{
  \PP_{\nu_T}( |R_0(s)| \geq a ) \leq \frac{C(g_0,\theta,t)}{a}. 
}
In particular, for $0<t \leq 1$,
\eqn{
  \PP_{\nu_T}( |R_0(s)| \geq a ) \leq \frac{C(g_0,\theta) t^{1/4}}{a}
}
holds.
\end{lemma}
%
%
\begin{proof}
The claim is obvious for $s=0$. We bound first for $s>0, a \geq 0$,
\eqan{
  \PP_{\nu_T}( R_0(s) \geq a )
  &=
  \frac{1}{\PP_{g_0}(\tau=\infty) T} \int_0^T \EE_{g_0}\!\left[ \1_{\{\tau=\infty\}} \PP_{u_r(\cdot+R_0(r))}(R_0(s) \geq a) \right] dr \\
  &\leq \frac{1}{\PP_{g_0}(\tau=\infty) T} \int_0^T \EE_{g_0}\!\left[ \1_{\{\tau=\infty\}} \PP\!\left( R_0\!\left( u_s^{*,l} \right) \geq a \right) \right] ds 
  = \PP\!\left( R_0\!\left( u_s^{*,l} \right) \geq a \right). \nn
}
Using Proposition~\ref{PRO:bound_on_u_star_l_t} and Lemma~\ref{LEM:bound_on_u_star_l_T} this yields for all $s>0$,
\eqn{
\lbeq{equ:bound_right_dom_wave}
  \PP_{\nu_T}( R_0(s) \geq a ) \leq \frac{1}{a} \EE\!\left[ 0 \vee R_0\!\left( u_s^{*,l} \right) \right] \leq \frac{C(\theta) \left( s \vee s^{1/4} \right)}{a}.
}
To prove the second half of the lemma, we follow the reasoning of the proof of \cite[Lemma~3.6]{T1996}. First observe that (by Lemma~\ref {LEM:exp_wavefront_bd} above the integrands are well-defined and Fubini's theorem can be applied)
\eqan{
  \EE_{\nu_T}[ R_0(s) ] 
  &= \frac{1}{\PP_{g_0}(\tau=\infty) T} \int_0^T \EE_{g_0}\!\left[ \1_{\{\tau=\infty\}} \EE_{u_r(\cdot+R_0(r))}[ R_0(s) ] \right] dr \\
  &= \frac{1}{\PP_{g_0}(\tau=\infty) T} \int_0^T \EE_{g_0}\!\left[ \1_{\{\tau=\infty\}} (R_0(r+s) - R_0(r)) \right] dr \nn\\
  &= \frac{1}{\PP_{g_0}(\tau=\infty) T} \left( \int_T^{T+s} \EE_{g_0}\!\left[ \1_{\{\tau=\infty\}} R_0(r) \right] dr - \int_0^s \EE_{g_0}\!\left[ \1_{\{\tau=\infty\}} R_0(r) \right] dr \right). \nn
}
Use
\eqn{
  \EE_{\nu_T}[ R_0(s) ]
  \leq \EE_{\nu_T}[ 0 \vee R_0(s) ] -a \PP_{\nu_T}( R_0(s) \leq -a)
}
and rearrange terms to conclude
\eqan{
\lbeq{equ:bound-nuT-1}
  & \PP_{\nu_T}( R_0(s)) \leq -a) \\
  & \leq \frac{1}{a} \left\{ \EE_{\nu_T}[ 0 \vee R_0(s) ] - \frac{1}{\PP_{g_0}(\tau=\infty) T} \left( \int_T^{T+s} \EE_{g_0}\!\left[ \1_{\{\tau=\infty\}} R_0(r) \right] dr - \int_0^s \EE_{g_0}\!\left[ \1_{\{\tau=\infty\}} R_0(r) \right] dr \right) \right\}. \nn
}
For the second and third term on the right hand side, Lemma~\ref{LEM:exp_wavefront_bd} yields 
\eqn{
  - \int_T^{T+s} \EE_{g_0}\!\left[ \1_{\{\tau=\infty\}} R_0(r) \right] dr \leq C(g_0,\theta) \int_T^{T+s} \left( 1 + \left( r \vee r^{1/4} \right) \right) dr
}
respectively 
\eqn{
\lbeq{equ:bound-nuT-2}
  \int_0^s \EE_{g_0}\!\left[ \1_{\{\tau=\infty\}} R_0(r) \right] dr \leq C(g_0,\theta) \int_0^s \left( 1 + \left( r \vee r^{1/4} \right) \right) dr. 
}
For the first term reason as in \eqref{equ:bound_right_dom_wave} to see that 
\eqan{
\lbeq{equ:bound-nuT-3}
  \EE_{\nu_T}[ 0 \vee R_0(s) ]
  \leq \EE[ 0 \vee R_0(u_s^{*,l}) ]  
  \leq C(\theta) \left( s \vee s^{1/4} \right). \nn
}
Collecting terms we get for $T \geq 1$,
\eqan{
  &\PP_{\nu_T}( R_0(s) \leq -a) \\
  &\leq \frac{1}{a} \left\{ C(\theta) \left( s \vee s^{1/4} \right) + C(g_0,\theta) \frac{1}{T} s \left( 1 + \left( (T+s) \vee (T+s)^{1/4} \right) \right) + C(g_0,\theta) \frac{1}{T} s \left( 1 + \left( s \vee s^{1/4} \right) \right) \right\} \nn\\
  &\leq \frac{1}{a} \left\{ C(\theta) \left( s \vee s^{1/4} \right) + C(g_0,\theta) s (1+s) + C(g_0,\theta) s \left( 1 + \left( s \vee s^{1/4} \right) \right) \right\} \nn
}
and the claim follows. 
\end{proof}
%

%
\section{Construction of travelling wave solutions arising from initial conditions with compact support: Proof of Theorem~\ref{THM:trav_wave_exists} and Proposition~\ref{PRO:limit_never_zero}}
\label{SEC:construction_tw}
%
%
%

We are now in a position to prove the analogue of \cite[Lemma~3.7]{T1996}.
%
%
\begin{lemma} \label{LEM:tightness}
If $\theta>\theta_c$ and $g_0 \in \SC_c^+ \backslash \{0\}$ then the sequence $\{ \nu_T: T \in \NN \}$ from Definition~\ref{DEF:nu_T} is tight.
\end{lemma}
%
%
\begin{proof}
The proof is similar to the proof of \cite[Lemma~3.7]{T1996}, except for the changes detailed below. To not confuse the reader in what follows, we note two small misprints in \cite{T1996} that are without influence on the rest of the proof. Namely, in the second and third line of the system of equations in the proof of Lemma~3.7, it should read $U(t,\cdot+R_1(t))$ respectively $U(t,\cdot+R_1(t-1))$ instead of $U(t,\cdot-R_1(t))$ respectively $U(t,\cdot-R_1(t-1))$. To adapt the proof to our setting, change all the wavefront markers from $R_1$ to $R_0$, initial conditions from $f_0$ to $g_0$, condition on the event $\{ \tau=\infty \}$ and proceed analogously to \cite{T1996} until one obtains terms $I$ and $II$. To bound term II, use that $\PP_{\nu_T}(|R_0(1)| \geq a) \leq C(g_0,\theta)/a$ by Lemma~\ref{LEM:3_6}. Term I can be bounded as in \cite{T1996}. Note that \cite{T1996} uses the definition of the wavefront marker $R_1$ to show that $\nu_T(\{ f: \langle f , \phi_1 \rangle \leq 1 ) = 1$. As we use $R_0$ instead, we proceed differently. 

Indeed, reason as above to conclude for $N \in \NN$ and $a, \delta>0$ arbitrary,
\eqan{
  \nu_T(\{ f: \langle f , \phi_1 \rangle > N \}) 
  \leq&\;  \PP_{\nu_T}(|R_0(\delta)| \geq a) + \delta/T \\
  &+ (\PP_{g_0}(\tau=\infty))^{-1} T^{-1} \int_\delta^T \PP_{g_0}\!\left( \PP_{u_{s-\delta}(\cdot+R_0(s-\delta))}\!\left( \langle u_\delta , \phi_1 \rangle > N e^{-a} \right) \right) ds. \nn
}
Choose $\delta$ small enough, then $a$ big enough to see that the first two terms can be made arbitrarily small, uniformly in $T \in \NN$. For the last term choose $N$ big enough and reason as in \eqref{cond_i_1}--\eqref{cond_i_2}.
\end{proof}
%

To prove the analogue of \cite[Theorem~3.8]{T1996}, that is Theorem~\ref{THM:trav_wave_exists}, we first need to prove a statement along the lines of \cite[(27)--(30)]{T1996}. The first property and the second part of the third property ($<\infty$) follow directly from the definition of $R_0$. The remaining properties are replaced by the statements in Proposition~\ref{PRO:limit_never_zero}. Before proving this proposition, we prove the following first.
%
%
\begin{lemma} \label{LEM:mass_somewhere_in_limit}
Let $\theta>\theta_c$ and $g_0 \in \SC_c^+ \backslash \{0\}$. Let $t \geq 0$ and $a, m>0$, $0<b \leq 1$ be arbitrarily fixed. Then 
\eqan{
  &\PP_{\nu_T}\!\left( \big\langle u_t(\cdot + R_0(u_t)) , \1_{(-2a,\infty)}(\cdot) \big\rangle < m \right) \\
   &\leq \left( \left( 1 - \frac{C(\theta) b^{1/4}}{a} \right) \vee 0 \right)^{-1} \left\{
  \frac{T+t}{T} \frac{C(g_0,\theta) b^{1/4}}{a} + \left( 1-e^{-2\theta \frac{m}{1-e^{-\theta b}}} \right) \right\} \nn
}
for all $T \in \NN$.
\end{lemma}
%
%
\begin{proof}
We have by Remark~\ref{RMK:def_nu_T}
\eqn{
  \PP_{\nu_{T+t}}( R_0(b) < -a ) 
  \geq \frac{1}{(T+t) \PP_{g_0}(\tau=\infty)} \int_0^T \EE_{g_0}\!\left[ \1_{\{\tau=\infty\}} \PP_{u_{s+t}(\cdot+R_0(s+t))}( R_0(b) < -a ) \right] ds. 
}
Monotonicity at time $s+t$ yields (recall the notation with non-continuous initial conditions from Notation~\ref{RMK:non-cont})
\eqan{
\lbeq{equ:show_wave_does_not_die_out} 
  & \PP_{\nu_{T+t}}( R_0(b) < -a ) \\
  &\geq \frac{1}{(T+t) \PP_{g_0}(\tau=\infty)} \int_0^T \EE_{g_0}\!\left[ \1_{\{\tau=\infty\}} \PP_{\1_{(-2a,\infty)}(\cdot) u_{s+t}(\cdot+R_0(s+t))} \left( \langle u_0 , 1 \rangle < m , \tau \leq b \right) \right. \nn \\
  &\quad \times \left. \PP_{\1_{(-\infty,-2a]}(\cdot) u_{s+t}(\cdot+R_0(s+t))} \left( R_0(b) < -a \right) \right] ds. \nn
}
The first probability in the product can be bounded below by 
\eqan{
\lbeq{equ:first-prob-b}
  & \PP_{\1_{(-2a,\infty)}(\cdot) u_{s+t}(\cdot+R_0(s+t))}( \langle u_0 , 1 \rangle < m , \tau \leq b ) \\
  &\geq \PP_{\1_{(-2a,\infty)}(\cdot) u_{s+t}(\cdot+R_0(s+t))}( \langle u_0 , 1 \rangle < m ) - \left( 1-e^{-2\theta \frac{m}{1-e^{-\theta b}}} \right), \nn
}
where we used \eqref{equ:bound_die_out_t} in the last line. For the second probability in the product in \eqref{equ:show_wave_does_not_die_out} we have
\eqn{
  \PP_{\1_{(-\infty,-2a]}(\cdot) u_{s+t}(\cdot+R_0(s+t))}( R_0(b) < -a ) 
  \geq \PP\!\left( R_0\!\left( u_b^{*,l} \right) < a \right)
  \geq \left( 1 - \frac{C(\theta) b^{1/4}}{a} \right) \vee 0
}
by Markov's inequality and Proposition~\ref{PRO:bound_on_u_star_l_t}. We obtain 
\eqan{
  & \PP_{\nu_{T+t}}( R_0(b) < -a ) \\
  &\geq \frac{\left( 1 - \frac{C(\theta) b^{1/4}}{a} \right) \vee 0}{(T+t) \PP_{g_0}(\tau=\infty)} \int_0^T \EE_{g_0}\!\left[ \1_{\{\tau=\infty\}} \PP_{\1_{(-2a,\infty)}(\cdot) u_{s+t}(\cdot+R_0(s+t))} \left( \langle u_0 , 1 \rangle < m \right) \right] ds \nn\\
  &\quad - \frac{\left( 1-e^{-2\theta \frac{m}{1-e^{-\theta b}}} \right)}{(T+t) \PP_{g_0}(\tau=\infty)} \int_0^T \EE_{g_0}\!\left[ \1_{\{\tau=\infty\}} \right] ds \nn\\
  &= \frac{T}{T+t} \left\{ \left( \left( 1 - \frac{C(\theta) b^{1/4}}{a} \right) \vee 0 \right) \PP_{\nu_T}\!\left( \big\langle u_t(\cdot + R_0(u_t)) , \1_{(-2a,\infty)}(\cdot) \big\rangle < m \right)  - \left( 1-e^{-2\theta \frac{m}{1-e^{-\theta b}}} \right) \right\}. \nn
}
Use Lemma~\ref{LEM:3_6} and rearrange terms to see that 
\eqan{
  &\PP_{\nu_T}\!\left( \big\langle u_t(\cdot + R_0(u_t)) , \1_{(-2a,\infty)}(\cdot) \big\rangle < m \right) \frac{T}{T+t} \left( \left( 1 - \frac{C(\theta) b^{1/4}}{a} \right) \vee 0 \right) \\
  &\leq \frac{C(g_0,\theta) b^{1/4}}{a} + \frac{T}{T+t} \left( 1-e^{-2\theta \frac{m}{1-e^{-\theta b}}} \right) \nn
}
which concludes the proof.
\end{proof}
%
%
We now have all the ingredients together to prove Proposition~\ref{PRO:limit_never_zero}. Recall that by Lemma~\ref{LEM:tightness} there exists a subsequence $\nu_{T_n}$ converging to some $\nu \in \CP(\SC_{tem}^+)$. 
%
%
\begin{proof}[Proof of Proposition~\ref{PRO:limit_never_zero}]
By definition of $\nu_T$ and $R_0(t)=R_0(u_t)$ we have $\nu_T(\{ f: R_0(f) = 0 \})=1$ for all $T \geq 1$. As the set $\{ f \in \SC_{tem}^+: R_0(f) \leq 0 \}$ is closed, 
\eqn{
  \nu(\{ f: R_0(f) > 0 \}) \leq \lim_{n \rightarrow \infty} \nu_{T_n}(\{ f: R_0(f)>0 \}) = 0
}
follows. 

Next let $t \geq 0$ be arbitrarily fixed. Let $m>0$ small and $A>0$ big. We obtain with $\phi_1(x) = \exp(-|x|)$ and as $\nu_{T_n} \Rightarrow \nu$ yields $\PP_{\nu_{T_n}} \Rightarrow \PP_\nu$ by Theorem~\ref{THM:tribe}, 
\eqan{
  \PP_\nu(u_t \equiv 0) 
  & \leq \PP_\nu( \langle u_t , \phi_1 \rangle < e^{-2A} m ) 
  \leq \liminf_{n \rightarrow \infty} \PP_{\nu_{T_n}}\!\left( \big\langle u_t , \phi_1 \big\rangle < e^{-2A} m \right) \\
  & \leq \liminf_{n \rightarrow \infty} \left\{ \PP_{\nu_{T_n}}\!\left( |R_0(t)| > A \right) + \PP_{\nu_{T_n}}\!\left( \big\langle u_t(\cdot + R_0(t)) , \1_{(-A,0)}(\cdot) \big\rangle < m , |R_0(t)| \leq A \right) \right\}. \nn
}
The first summand can be bounded by Lemma~\ref{LEM:3_6} and we get as a result
\eqn{
  \PP_\nu(u_t \equiv 0) \leq \frac{C(g_0,\theta,t)}{A} + \liminf_{n \rightarrow \infty} \PP_{\nu_{T_n}} \!\left( \big\langle u_t(\cdot + R_0(t)) , \1_{(-A,0)}(\cdot) \big\rangle < m \right). 
}
Choose $b=1$, $m$ small enough and $A$ big enough in Lemma~\ref{LEM:mass_somewhere_in_limit} to see that the second summand becomes arbitrarily small. By further increasing $A$ the first summand becomes arbitrarily small, too. This proves the second half of the proposition. Moreover, $\nu(\{ f: -\infty < R_0(f) \leq 0 \}) = 1$ follows. 

Consider the case $t=0$. Let $a>0$ arbitrarily small and $M, \lambda>0$ big. Let $\phi_{a,\lambda}(x) \equiv \exp(-\lambda |x+a|)$, then
\eqn{
  \PP_\nu( R_0(0)  \leq -a )
  \leq \PP_\nu\!\left( \big\langle u_0 , \phi_{a,1} \big\rangle > M \right) + \PP_\nu\!\left( \big\langle u_0 , \phi_{0,\lambda} \big\rangle < e^{-\lambda a} M \right).
}
The first term can be bounded using Corollary~\ref{COR:coupling_dominated_by_upper} and reasoning as in \eqref{equ:bound_expectation_u_star-calc} by
\eqn{
  \PP_\nu\!\left( \big\langle u_0 , \phi_{a,1} \big\rangle > M \right) 
  \leq M^{-1} \theta \langle \phi_{a,1} , 1 \rangle.
}
This yields 
\eqn{
  \PP_\nu( R_0(0)  \leq -a )
  \leq \frac{2 \theta}{M} + \liminf_{n \rightarrow \infty} \PP_{\nu_{T_n}}\!\left( \big\langle u_0 , \1_{(-a/2,0)}(\cdot) \big\rangle < e^{-\lambda a/2} M \right).
}
Choose $M$ big enough to make the first term small. Set $m \equiv \exp(-\lambda a/2) M$. Note that for $a, M$ fixed, $m$ can be made arbitrarily small by choosing $\lambda$ arbitrarily big. To make the second summand arbitrarily small, apply Lemma~\ref{LEM:mass_somewhere_in_limit} by first choosing $b$ small enough such that $b^{1/4}/a$ is small and afterwards choosing $m$ small enough. Hence, for all $a>0$, $\nu(\{ f: R_0(f) \leq -a \}) = \PP_\nu( R_0(0)  \leq -a )=0$ and the remaining claim follows. 
\end{proof}
%
%
\noindent\textit{Proof of Theorem~\ref{THM:trav_wave_exists}.}
We first introduce a set of approximating wavefront-markers tailored to $R_0(f)$. Note that $R_0(f)$ is not continuous on $\SC_{tem}^+$. In what follows let $m>0$ and $N \in \NN$ be arbitrarily fixed. For $f \in \SC_{tem}^+$ set
\eqn{
  R^{m,N}(f) = R^{m,N}(f \1_{[-N,N]}) \equiv \sup\!\left\{ x \in [-N,N]: f(x)>0 \mbox{ and } \big\langle f , 1(x < \cdot \leq N) \big\rangle \geq m \right\}
}
with the convention that $\sup \emptyset = -N$ in the above. We note that $m \mapsto R^{m,N}(f) \in [-N,N]$ for all $f \in \SC_{tem}^+$ and 
\eqn{
  R^{m,N}(f) \uparrow R_0\big( f \1_{[-N,N]} \big) \vee (-N) =: R^{0,N}(f) \mbox{ as } m \downarrow 0^+.
}
Let $\Phi \in \SC_c^+$ be a fixed smooth function supported on $(-1,0)$ such that $\int \Phi(x) dx = 1$, use $\cstar$ to denote convolution of functions and set $\Phi_{m_0}(x) \equiv (1/m_0) \Phi(x/m_0)$ for $m_0>0$ fixed. Finally set
\eqn{
  R_{m_0}^N(f) \equiv \left( \Phi_{m_0} \cstar R^{\LargerCdot,N}(f) \right)\!(0) = \int_0^{m_0} \Phi_{m_0}(-m) R^{m,N}(f) dm.
} 
This wavefront marker is continuous on $\SC_{tem}^+$ and takes values in $[-N,N]$. By \cite[Theorem~8.15]{bF1999}, $R_{m_0}^N(f) \rightarrow R^{0,N}(f)$ for $m_0 \downarrow 0^+$. By definition of $\Phi$ and $R^{m,N}(f)$ we further have $R^{m,N}(f) \leq R_{m_0}^N(f) \leq R^{0,N}(f)$ for all $m \geq m_0$. 

In what follows let $t \geq 0, \epsilon>0$ be arbitrarily fixed. By Lemma~\ref{LEM:3_6} and Proposition~\ref{PRO:limit_never_zero}, there exists $N=N(t,\epsilon) \in \NN$ big enough such that
\eqn{
\lbeq{stationary-1}
  \sup_{T \in \NN} \PP_{\nu_T}\!\left( R_0(u_t) \neq R^{0,N}(u_t) \right) + \PP_\nu\!\left( R_0(u_t) \neq R^{0,N}(u_t) \right) < \epsilon.
}
By Lemma~\ref{LEM:mass_somewhere_in_limit}, Proposition~\ref{PRO:limit_never_zero} and the definition of $R_{m_0}^N(f)$, for all $\delta>0$ there exists $m_0=m_0(t,\epsilon,N,\delta)>0$ small enough such that
\eqn{
\lbeq{stationary-2}
  \sup_{T \in \NN} \PP_{\nu_T}\!\left( 0 \leq R^{0,N}(u_t) - R_{m_0}^N(u_t) < \delta \right) + \PP_\nu\!\left( 0 \leq R^{0,N}(u_t) - R_{m_0}^N(u_t) < \delta \right) < \epsilon.
}
Let $\nu_{T_n}$ be a subsequence that converges to $\nu$. Then Theorem~\ref{THM:tribe} yields $\PP_{\nu_{T_n}} \Rightarrow \PP_\nu$. Hence, there exists a compact set $K=K(t,\epsilon) \subset \SC_{tem}^+$ big enough such that
\eqn{
\lbeq{stationary-3}
  \sup_{n \in \NN} \PP_{\nu_{T_n}}( u_t \not\in K) + \PP_\nu( u_t \not\in K) < \epsilon.
}
Let $F: \SC_{tem}^+ \rightarrow \RR$ be an arbitrarily fixed bounded and continuous function. By the characterization of compact subsets of $\SC_{tem}^+$ (see for instance the introduction of \cite{T1996}), there exists $\delta=\delta(K,F,N)>0$ small enough such that
\eqn{
\lbeq{stationary-4}
  \sup_{0 \leq |a| \leq \delta} \sup_{0 \leq |b| \leq N} \sup_{f \in K}| F(f(\cdot+b)) - F(f(\cdot + b + a))| < \epsilon.
}

To complete the proof, let $t \geq 0, \epsilon>0$ and $F: \SC_{tem}^+ \rightarrow \RR$ bounded and continuous be arbitrarily fixed. Choose $N \in \NN$ big enough such that \eqref{stationary-1} holds and compact $K \subset \SC_{tem}^+$ big enough such that \eqref{stationary-3} holds. Then choose $\delta>0$ small enough such that \eqref{stationary-4} holds and subsequently $m_0$ small enough such that \eqref{stationary-2} holds. From \eqref{stationary-1}--\eqref{stationary-4} we conclude that 
\eqan{
  & \sup_{n \in \NN} \EE_{\nu_{T_n}}\!\left[ \left| F(u_t(\cdot + R_0(u_t))) - F(u_t(\cdot + R_{m_0}^N(u_t))) \right| \right] \lbeq{stationary-5} \\
  & + \EE_\nu\!\left[ \left| F(u_t(\cdot + R_0(u_t))) - F(u_t(\cdot + R_{m_0}^N(u_t))) \right| \right] < 3 \epsilon \normx{F}{\infty} + \epsilon \equiv \epsilon(F) \nn
}
holds. By the continuity of $R_{m_0}^N(f)$, we further have for all $t \geq 0$ fixed,
\eqn{
  \big| \EE_{\nu_{T_n}}\!\left[ F(u_t(\cdot - R_{m_0}^N(u_t))) \right] - \EE_\nu\!\left[ F(u_t(\cdot - R_{m_0}^N(u_t))) \right] \big| \rightarrow 0 \mbox{ for } n \rightarrow \infty. 
}
Together with \eqref{stationary-5} this yields 
\eqan{
  & \left| \EE_\nu\!\left[ F(u_t(\cdot + R_0(u_t))) \right] - \nu(F) \right| \\
  & \leq \epsilon(F) + \left| \EE_\nu\!\left[ F(u_t(\cdot + R_{m_0}^N(u_t))) \right] - \nu(F) \right| \nn\\
  & = \epsilon(F) + \left| \lim_{n \rightarrow \infty} \EE_{\nu_{T_n}}\!\left[ F(u_t(\cdot + R_{m_0}^N(u_t))) \right] - \nu(F) \right| \nn\\
  & \leq 2\epsilon(F) + \left| \lim_{n \rightarrow \infty} \EE_{\nu_{T_n}}\!\left[ F(u_t(\cdot + R_0(u_t))) \right] - \nu(F) \right| \nn\\
  & = 2\epsilon(F) + \left| \lim_{n \rightarrow \infty} (\PP_{g_0}(\tau=\infty) T_n)^{-1} \int_0^{T_n} \EE_{g_0}\!\left[ \1_{\{\tau=\infty\}} F(u_{s+t}(\cdot + R_0(u_{s+t}))) \right] ds - \nu(F) \right| \nn\\
  & = 2\epsilon(F) + \left| \lim_{n \rightarrow \infty} (\PP_{g_0}(\tau=\infty) T_n)^{-1} \int_0^{T_n} \EE_{g_0}\!\left[ \1_{\{\tau=\infty\}} F(u_s(\cdot + R_0(u_s))) \right] ds - \nu(F) \right| \nn\\
  & = 2\epsilon(F), \nn
}
where in the second inequality we used that both limits exist. Take $\epsilon \rightarrow 0^+$ to see that under $\PP_\nu$ the one-dimensional marginals of $\big( u_t(\cdot + R_0(u_t)) \big)_{t \geq 0}$ have law $\nu$. It is straightforward to check that the process is also Markov. The process is therefore stationary in time and with Proposition~\ref{PRO:limit_never_zero} the claim follows.
\qed
%
%
\section{Recurrence}
\label{SEC:recurrence}
%
%
\noindent\textit{Proof of Theorem~\ref{THM:coming_back}.}
Let $\{\nu_T\}_{T \in \NN}$ be as in Definition~\ref{DEF:nu_T} and $\nu_{T_k}$ a subsequence that converges to $\nu$ for $k \rightarrow \infty$ as given by Lemma~\ref{LEM:tightness}. For $M>0, \phi_0 \in \SC_c^+$ let 
\eqn{
  A_{M,\phi_0} \equiv \{ \phi \in \SC_{tem}^+: \exists |x| \leq M \mbox{ such that } \phi > \phi_0(\cdot - x) \} \subset \SC_{tem}^+. 
}
By tightness of $\{\nu_{T_k}\}_{k \in \NN}$ and Proposition~\ref{PRO:limit_never_zero}, for all $\epsilon>0$ there exist $M>0$ big, $\phi_0 \in \SC_c^+ \backslash \{0\}$ small and $k_0 \in \NN$ big enough such that
\eqn{
  \nu_{T_k}( A_{M,\phi_0} ) > 1-\epsilon, \qquad \forall k \geq k_0. 
}
The definition of $\nu_{T_k}$, Fubini-Tonelli's theorem and the analogue for the travelling wave to the left of the support yield by means of a proof by contradiction, that $M>0$, $\phi_0 \in \SC_c^+ \backslash \{0\}$ can be further chosen such that
\eqn{
\lbeq{equ:bigger-phi-0}
  \PP_{g_0}\!\left( \forall n \in \NN \ \exists t \geq n: u_t(\cdot + R_0(t)) \in A_{M,\phi_0} \mbox{ and } u_t(\cdot + L_0(t)) \in A_{M,\phi_0} \;|\; \tau=\infty \right) > 1-2\epsilon.
}
Before we continue with the proof of Theorem~\ref{THM:coming_back}, we establish the following result first.
%
%
\begin{lemma}
\label{LEM:for-all-outside-dom}
Let $\theta>\theta_c$ and $\psi_0 \in \SC_c^+$, then
\eqn{
\lbeq{equ:lemma-for-all-outside-dom}
  \PP_{g_0}\!\left( \forall n \in \NN \ \exists t \geq n: u_{t+1}(\cdot + R_0(t)) \geq \psi_0 \mbox{ and } u_{t+1}(\cdot + L_0(t)) \geq \psi_0 \;|\; \tau=\infty \right) = 1. 
}
\end{lemma}
%
%
\noindent\textit{Proof of Lemma~\ref{LEM:for-all-outside-dom}.}
First observe that for $M>0, \phi_0 \in \SC_c^+ \backslash \{0\}$ arbitrary, there exists $\delta=\delta(M,\phi_0,\psi_0)>0$ such that
\eqn{
\lbeq{equ:move-from-small-to-big}
  \inf_{|x| \leq M} \PP_{\phi_0(\cdot - x)}( u_1 \geq \psi_0) > 2\delta. 
}
The idea of the proof of \eqref{equ:lemma-for-all-outside-dom} is a geometric series type of argument. To this goal, set $\tau_0 \equiv 0$ and
\eqn{
  \tau_n = \tau_n(M,\phi_0) \equiv \inf_{t>\tau_{n-1}+1} \{ u_t(\cdot + R_0(t)) \in A_{M,\phi_0} \mbox{ and } u_t(\cdot + L_0(t)) \in A_{M,\phi_0} \}, \quad n \in \NN 
}
with the convention that $\tau_n \equiv -\infty$ if $\tau_{n-1}=-\infty$ or if the infimum is taken over an empty set. Let $\epsilon>0$ be arbitrarily fixed. By \eqref{equ:bigger-phi-0}, we can choose $M>0, \phi_0 \in \SC_c^+ \backslash \{0\}$ such that
\eqn{
\lbeq{equ:tau-infty}
  \PP_{g_0}( \lim_{n \rightarrow \infty} \tau_n = \infty \;|\; \tau=\infty ) > 1-2\epsilon.
}
Let $D>0$ be arbitrary, to be chosen later on. By Proposition~\ref{PRO:overall_mass_and_support} and the compact support property, there exists $I_0 \in \NN$ big enough such that
\eqn{
\lbeq{equ:tau-infty-d}
  \PP_{g_0}( \{ \lim_{n \rightarrow \infty} \tau_n = \infty \} \cap \{ R_0(\tau_i)-L_0(\tau_i) \geq D, \forall i \geq I_0 \} \;|\; \tau=\infty ) > 1-3\epsilon.
}
Next fix $K, I \in \NN$ arbitrary with $I \geq I_0$. Condition on $\SF_{\tau_{I+K}}$ to get
\eqan{
  & \PP_{g_0}\Big( \bigcap_{n \in \{ I+1,\ldots,I+K\}} \Big( \left\{ \tau_n > -\infty \right\} \cap \left\{ R_0(\tau_n)-L_0(\tau_n) \geq D \right\} \lbeq{rec-lemma-1} \\
  & \qquad\qquad\qquad\qquad\qquad \cap \big\{ u_{\tau_n+1}(\cdot + R_0(\tau_n)) \wedge u_{\tau_n+1}(\cdot + L_0(\tau_n)) \geq \psi_0 \big\}^c \Big) \;\big|\; \tau=\infty \Big) \nn \\
  & \leq \tfrac{1}{\PP_{g_0}(\tau=\infty)} \EE_{g_0}\!\left[ \1_{\{ \tau_{I+1} > -\infty \}} \1_{\{ R_0(\tau_{I+1}) - L_0(\tau_{I+1}) \geq D \}} \1_{\big\{ u_{\tau_{I+1}+1}(\cdot + R_0(\tau_{I+1})) \wedge u_{\tau_{I+1}+1}(\cdot + L_0(\tau_{I+1})) \geq \psi_0 \big\}^c} \times \cdots \right. \nn\\
  & \quad \times \1_{\{ \tau_{I+K-1} > -\infty \}} \1_{\{ R_0(\tau_{I+K-1}) - L_0(\tau_{I+K-1}) \geq D \}} \1_{\big\{ u_{\tau_{I+K-1}+1}(\cdot + R_0(\tau_{I+K-1})) \wedge u_{\tau_{I+K-1}+1}(\cdot + L_0(\tau_{I+K-1})) \geq \psi_0 \big\}^c}  \nn\\
  & \quad \times \left. \1_{\{ \tau_{I+K} > -\infty \}} \1_{\{ R_0(\tau_{I+K}) - L_0(\tau_{I+K}) \geq D \}} \EE_{g_0}\!\left[ \1_{\big\{ u_{\tau_{I+K}+1}(\cdot + R_0(\tau_{I+K})) \wedge u_{\tau_{I+K}+1}(\cdot + L_0(\tau_{I+K})) \geq \psi_0 \big\}^c} \;\big|\; \SF_{\tau_{I+K}} \right] \right]. \nn
}

Reason as in \cite[Lemma~2.1.7]{MT1994} to see that for $D \in \RR, D>R_0(\phi_0)-L_0(\phi_0)$, there exists a coupling such that $u^{(\phi_0)}, u^{(\phi_0(\cdot-D))}, u^{(\phi_0+\phi_0(\cdot-D))}$ are solutions to \eqref{equ:SPDE}, starting in $\phi_0$, $\phi_0(\cdot-D)$ respectively $\phi_0+\phi_0(\cdot-D)$ such that $u^{(\phi_0)}$, $u^{(\phi_0(\cdot-D))}$ are independent and
\eqn{
  u^{(\phi_0)}_t + u^{(\phi_0(\cdot-D))}_t = u^{(\phi_0+\phi_0(\cdot-D))}_t \mbox{ for } t \leq \inf\!\left\{s \geq 0: R_0\!\left(u^{(\phi_0)}_s\right)>L_0\!\left(u^{(\phi_0(\cdot-D))}_s\right) \right\}.
}
By the compact support property and \eqref{equ:move-from-small-to-big} it follows that there exists $D=D(M,\phi_0,\psi_0,\delta)>0$ big enough such that
\eqn{
  \inf_{D' \geq D} \inf_{|x_l|, |x_r| \leq M} \EE_{\phi_0(\cdot-x_l)+\phi_0(\cdot-D'-x_r)}\Big[ \1_{\big\{ u_1 \geq \psi_0+\psi_0(\cdot-D') \big\}} \Big]
  \geq \delta^2.
}
By monotonicity in the initial condition we obtain, for $I \in \NN$ big enough, as an upper bound to the last term of the right hand side of \eqref{rec-lemma-1},
\eqan{
  & \1_{\{ \tau_{I+K} > -\infty \}} \1_{\{ R_0(\tau_{I+K}) - L_0(\tau_{I+K}) \geq D \}} \EE_{g_0}\!\left[ \1_{\big\{ u_{\tau_{I+K}+1}(\cdot + R_0(\tau_{I+K})) \wedge u_{\tau_{I+K}+1}(\cdot + L_0(\tau_{I+K})) \geq \psi_0 \big\}^c} \;\big|\; \SF_{\tau_{I+K}} \right] \\
  & < 1-\delta^2. \nn
}
Iteration of the argument results in the upper bound $\tfrac{(1-\delta^2)^K}{\PP_{g_0}(\tau=\infty)}$ to \eqref{rec-lemma-1}.

By \eqref{equ:tau-infty-d} and this upper bound, we finally have for $K \in \NN$ arbitrarily fixed,
\eqan{
  & \PP_{g_0}\big( \big\{ \forall n \in \NN \ \exists t \geq n: u_{t+1}(\cdot + R_0(t)) \geq \psi_0 \mbox{ and } u_{t+1}(\cdot + L_0(t)) \geq \psi_0 \big\}^c \;\big|\; \tau=\infty \big) \\
  & \leq 3\epsilon + \lim_{I \rightarrow \infty} \PP_{g_0}\Big( \bigcap_{n \in \{ I+1,\ldots,I+K\}} \Big( \left\{ \tau_n > -\infty \right\} \cap \left\{ R_0(\tau_n)-L_0(\tau_n) \geq D \right\} \nn \\
  & \qquad\qquad\qquad\qquad\qquad \cap \big\{ u_{\tau_n+1}(\cdot + R_0(\tau_n)) \wedge u_{\tau_n+1}(\cdot + L_0(\tau_n)) \geq \psi_0 \big\}^c \Big) \;\big|\; \tau=\infty \Big) \nn \\
  & \leq 3\epsilon + (1-\delta^2)^K / \PP_{g_0}( \tau=\infty ). \nn
}
Choose $K \rightarrow \infty$ and let $\epsilon \downarrow 0^+$ to conclude the claim.
\qed\\
%
%

\noindent\textit{Continuation of the proof of Theorem~\ref{THM:coming_back}.}
Let $\psi_0 \in \SC_c^+ \backslash \{0\}$ arbitrary. Set $\tilde{\tau}_0 \equiv 0$ and
\eqn{
\lbeq{equ:def-tilde-tau}
  \tilde{\tau}_n = \tilde{\tau}_n^{(\psi_0)} \equiv \inf_{t>\tilde{\tau}_{n-1}+1} \{ u_{t+1}(\cdot + R_0(t)) \geq \psi_0 \mbox{ and } u_{t+1}(\cdot + L_0(t)) \geq \psi_0 \}, \quad n \in \NN
}
with the convention that $\tilde{\tau}_n \equiv -\infty$ if $\tilde{\tau}_{n-1}=-\infty$ or if the infimum is taken over an empty set. We remark at this point already that $\tilde{\tau}_n$ is not a stopping time itself but that $\tilde{\tau}_n+1$ is. By Proposition~\ref{PRO:overall_mass_and_support} and Lemma~\ref{LEM:for-all-outside-dom},
\eqn{
\lbeq{equ:def-epsilon-prime}
  \PP_{g_0}\!\left( \lim_{n \rightarrow \infty} [ R_0(\tilde{\tau}_n) \vee (-L_0(\tilde{\tau}_n)) ] = \infty \;|\; \tau=\infty \right) 
  \geq \PP_{g_0}( \lim_{n \rightarrow \infty} \tilde{\tau}_n = \infty \;|\; \tau=\infty ) 
  = 1
}
follows, where we set $R_0(-\infty) = -L_0(-\infty) \equiv -\infty$ and note that $\tilde{\tau}_{n+1}-\tilde{\tau}_n \geq 1$ as long as $-\infty<\tilde{\tau}_{n+1}$. 

We complete the proof with the help of the following lemma. Its proof follows below.
%
%
\begin{lemma}
\label{LEM:help-1}
Let $\theta>\theta_c$. For arbitrary $\tilde{\epsilon}>0$ there exists $\psi_0=\psi_0(\tilde{\epsilon}) \in \SC_c^+$ with $R_0(\psi_0)=0$ such that
\eqn{
\lbeq{equ:limsup-inf-prob-1}
  \mbox{ for } \quad B(K) \equiv \{ \exists t>0: R_0(t)>K \} \quad \mbox{ we have } \quad \PP_{\psi_0}( B(K)^c \;|\; \tau=\infty ) 
  \leq 16 \tilde{\epsilon}, \qquad \forall\ K \geq 0.
}
\end{lemma}
%
%
We only prove recurrence, that is property \eqref{equ:def-recurrent}, in case $B=(-1,1)$. The proof for general open $B \subset \RR$ is analoguous. Let $g_0 \in \SC_c^+ \backslash \{0\}$ and $T>0$ be arbitrarily fixed. Recall that solutions to \eqref{equ:SPDE} take values in $\SC([0,\infty),\SC_{tem}^+)$. We show in what follows that
\eqn{
  \PP_{g_0}\!\left( \exists t \geq T: \mbox{supp}(u(t,\cdot)) \cap (-1,1) \neq \emptyset \ \big| \ \tau=\infty \right) \geq 1-16 \tilde{\epsilon} \ \mbox{ for all } \ \tilde{\epsilon}>0.
}
Let $\tilde{\epsilon}>0$ be arbitrarily fixed and $\psi_0=\psi_0(\tilde{\epsilon})$ as in Lemma~\ref{LEM:help-1}. By \eqref{equ:def-epsilon-prime}, conditional on the event $\{\tau=\infty\}$, with probability one, there exists (a random) $n \in \NN$ such that $\tilde{\tau}_n=\tilde{\tau}_n^{(\psi_0)} \geq T$. Assume without loss of generality that $L_0(\tilde{\tau}_n)=L_0(u_{\tilde{\tau}_n}^{(g_0)}) < L_0(\psi_0)$ (recall that $R_0(\psi_0)=0$). The other case follows by similar reasoning. By definition of $\tilde{\tau}_n$, $h(x) \equiv u_{\tilde{\tau}_n+1}^{(g_0)}(x+ L_0(\tilde{\tau}_n)) \geq \psi_0(x)$. We can therefore construct a coupling of solutions to \eqref{equ:SPDE} satisfying 
\eqn{
\lbeq{equ:coupling-1}
  u_{\tilde{\tau}_n+1+t}^{(g_0)}(x+ L_0(\tilde{\tau}_n)) = u_t^{(h)}(x) \geq u_t^{(\psi_0)}(x) 
}
for all $t \geq 0, x \in \RR$ almost surely. From Lemma~\ref{LEM:help-1} it follows that conditional on the event $\{ \tau^{(\psi_0)} = \infty \}$, there exists (a random) $t_0 > 0$ such that $R_0(u_{t_0}^{(\psi_0)}) > -L_0(\tilde{\tau}_n)$ with probability at least $1-16 \tilde{\epsilon}$. Thus, by the coupling in \eqref{equ:coupling-1}, there exists (a random) $s>0$ with $\tilde{\tau}_n + 1 \leq s \leq \tilde{\tau}_n + 1 + t_0$ such that $\mbox{supp}\big( u_s^{(g_0)} \big) \cap (-1,1) \neq \emptyset$.

In case $u_\cdot^{(\psi_0)}$ dies out after a time-period of length $\tau^{(\psi_0)}$, we restart the argument at a time $\tilde{\tau}_m^{(g_0)}>\tilde{\tau}_n^{(g_0)}+\tau^{(\psi_0)}, m>n$. Another geometric series type of argument concludes the claim.
\qed\\
%
%

\noindent\textit{Proof of Lemma~\ref{LEM:help-1}.}
For $\tilde{\epsilon}>0$ arbitrarily fixed, we will prove the existence of $\psi_0 \in \SC_c^+$, symmetric around zero, such that \eqref{equ:limsup-inf-prob-1} holds for all $K \geq R_0(\psi_0)$. The claim then follows by the translation invariance of solutions to \eqref{equ:SPDE}. 

Consider arbitrary $g_0=\psi_0$, for now independent of $\tilde{\epsilon}$ (only at the end we shall choose $\psi_0=\psi_0(\tilde{\epsilon})$) but symmetric around zero. Set
\eqn{
  A \equiv A(\psi_0) \equiv \left\{ \limsup_{n \rightarrow \infty} R_0\big( \tilde{\tau}_n^{(\psi_0)} \big) < \infty \right\}.
}
As $g_0=\psi_0$ is an initial condition that is symmetric around zero, we obtain by \eqref{equ:def-epsilon-prime} and symmetry, 
\eqn{
\lbeq{equ:limsup-inf-prob-12}
  \PP_{\psi_0}(A^c \;|\; \tau=\infty ) = \PP_{\psi_0}\!\left( \limsup_{n \rightarrow \infty} R_0(\tilde{\tau}_n) = \infty \;\Big|\; \tau=\infty \right) = \PP_{\psi_0}\!\left( \limsup_{n \rightarrow \infty} (-L_0(\tilde{\tau}_n)) = \infty \;\Big|\; \tau=\infty \right) \geq \frac{1}{2}. 
}

To establish \eqref{equ:limsup-inf-prob-1}, first observe that there exist $M_0=M_0(\tilde{\epsilon},\psi_0) > 0$ and $N_0 = N_0(\tilde{\epsilon},\psi_0,M_0) \in \NN$ big enough such that with
\eqn{
  A' = A'(\psi_0) \equiv \left\{ \sup_{1 \leq n \leq N_0} R_0(\tilde{\tau}_n) < M_0 \right\} 
}
we have (for two sets $A,B$ denote symmetric difference by $A \triangle B = (A \cap B^c) \cup (A^c \cap B)$)
\eqn{
\lbeq{equ:inf-to-fin}
  \PP_{\psi_0}\!\left( A \triangle A' \;|\; \tau=\infty \right) < \tilde{\epsilon}. 
}
Choose $M_1=M_1(\tilde{\epsilon},\psi_0,M_0,N_0,L_0(\psi_0))$ big enough such that
\eqn{
\lbeq{equ:restriction-on-av-rhs}
  \PP_{\psi_0}\!\left( \sup_{1 \leq n \leq N_0} |R_0(\tilde{\tau}_n)| < M_1 \;\Big|\; \tau=\infty \right) > 1-\tilde{\epsilon}.
}
Indeed, apply \eqref{equ:def-epsilon-prime} with $g_0=\psi_0$ and Lemma~\ref{LEM:exp_wavefront_bd}. Note that for all $K \in \RR$, $B(K)^c \subseteq A(\psi_0)$ holds. Together with \eqref{equ:inf-to-fin} this yields
\eqn{
  \PP_{\psi_0} (B(K)^c \;|\; \tau=\infty) \leq \PP_{\psi_0} (B(K)^c \cap A' \;|\; \tau=\infty) + \tilde{\epsilon}. 
}
By \eqref{equ:restriction-on-av-rhs},
\eqn{
  \PP_{\psi_0} (B(K)^c \cap A' \;|\; \tau=\infty) 
  \leq \tilde{\epsilon} + \PP_{\psi_0}( B(K)^c \cap A' \cap \{ |R_0(\tilde{\tau}_{N_0})| < M_1 \} \;|\; \tau=\infty ).
}
Condition on $\SF_{\tilde{\tau}_{N_0}+1}$ (recall that $\tilde{\tau}_n+1$ is a stopping time) to obtain with the help of \eqref{equ:def-epsilon-prime},
\eqan{
  \PP_{\psi_0} (B(K)^c \cap A' \cap \{ |R_0(\tilde{\tau}_{N_0})| < M_1 \} \;|\; \tau=\infty) 
  &= \frac{ \EE_{\psi_0}\!\left[ \EE_{\psi_0}\!\left[ \1_{B(K)^c \cap A' \cap \{ |R_0(\tilde{\tau}_{N_0})| < M_1 \}} \1_{\{\tau=\infty\}} \;\big|\; \SF_{\tilde{\tau}_{N_0}+1} \right] \right] }{ \PP_{\psi_0}( \tau=\infty ) } \\
  &\leq \frac{ \EE_{\psi_0}\!\left[ \1_{A' \cap \{ |R_0(\tilde{\tau}_{N_0})| < M_1 \} } \1_{\{ \tilde{\tau}_{N_0}>-\infty \}} \cdot \PP_{u_{\tilde{\tau}_{N_0}+1}} ( B(K)^c ) \right] }{ \PP_{\psi_0}( \tau=\infty) }. \nn
}
By definition of $\tilde{\tau}_n$, $u_{\tilde{\tau}_{N_0}+1}(\cdot + R_0(\tilde{\tau}_{N_0})) \geq \psi_0$ on $\{ \tilde{\tau}_{N_0} > -\infty \}$. Monotonicity therefore yields
\eqn{
  \PP_{\psi_0} (B(K)^c \cap A' \cap \{ |R_0(\tilde{\tau}_{N_0})| < M_1 \} \;|\; \tau=\infty) 
  \leq \frac{ \EE_{\psi_0}\!\left[ \1_{A' \cap \{ |R_0(\tilde{\tau}_{N_0})| < M_1 \}} \right] \PP_{\psi_0} ( B(K+M_1)^c ) }{ \PP_{\psi_0}( \tau=\infty ) }
}
and we conclude
\eqn{
  \PP_{\psi_0} (B(K)^c \;|\; \tau=\infty) 
\leq 2 \tilde{\epsilon} + \PP_{\psi_0} ( A' ) \PP_{\psi_0} ( B(K+M_1)^c  \;|\; \tau=\infty ) + \PP_{\psi_0}( \tau<\infty) / \PP_{\psi_0}( \tau=\infty ).
}
Recall that  $B(K+nM_1)^c \subseteq A(\psi_0)$ for $n \in \NN$ arbitrary. Iterate the above argument to finally obtain
\eqan{
  \PP_{\psi_0} (B(K)^c \;|\; \tau=\infty) 
\leq& 2 \tilde{\epsilon} + \PP_{\psi_0}( \tau<\infty) / \PP_{\psi_0}( \tau=\infty ) \\
& + \PP_{\psi_0} ( A' ) \big\{ 2 \tilde{\epsilon} + \PP_{\psi_0}( \tau<\infty) / \PP_{\psi_0}( \tau=\infty ) \nn\\
& + \PP_{\psi_0} ( A' ) \big[ 2 \tilde{\epsilon} + \PP_{\psi_0}( \tau<\infty) / \PP_{\psi_0}( \tau=\infty )
 + \cdots \nn\\
 & + \PP_{\psi_0} ( A' ) \PP_{\psi_0}\big( B(K+n M_1)^c \;|\; \tau=\infty \big) \cdots \big] \big\} \lbeq{equ:last-geom-sum-arg} \nn\\
  \leq& \left( 2 \tilde{\epsilon} + \PP_{\psi_0}( \tau<\infty) / \PP_{\psi_0}( \tau=\infty ) \right) \sum_{k=0}^{n-1} (\PP_{\psi_0}( A' ))^k + (\PP_{\psi_0}( A' ))^n. \nn
}
Using \eqref{equ:limsup-inf-prob-12} and \eqref{equ:inf-to-fin}, conclude that 
\eqan{
  \PP_{\psi_0}(A')
  &\leq \PP_{\psi_0}(A' \;|\; \tau=\infty) \PP_{\psi_0}(\tau=\infty) + \PP_{\psi_0}(\tau<\infty) \\
  &\leq \left( \PP_{\psi_0}(A \;|\; \tau=\infty) + \PP_{\psi_0}\!\left( A \triangle A' \;|\; \tau=\infty \right) \right) \PP_{\psi_0}(\tau=\infty) + \PP_{\psi_0}(\tau<\infty) \nn\\
  &< \left( \tfrac{1}{2}+\tilde{\epsilon} \right) \PP_{\psi_0}(\tau=\infty) + \PP_{\psi_0}(\tau<\infty). \nn
}
Finally choose $\psi_0=\psi_0(\tilde{\epsilon})$ such that $\PP_{\psi_0}(\tau<\infty) \leq \tilde{\epsilon}$. For $\tilde{\epsilon}$ small, this yields $\PP_{\psi_0}(A') < 3/4$ and
\eqn{
  \PP_{\psi_0} (B(K)^c \;|\; \tau=\infty) \leq 16 \tilde{\epsilon}
}
follows for $\tilde{\epsilon}$ small enough. This completes the proof.
\qed
%
%
%
%
\section{Appendix}
%
%
To clarify the comment at the end of \cite[Lemma~2.1]{T1996}, observe the following:
\begin{remark} \label{RMK:tribe_lemma_2_1}
Let $u$ be a solution to \eqref{equ:SPDE} started at $u_0 \in \SC_{tem}^+$. Suppose for some $R>0$ that $u_0$ is supported outside $(R-2,\infty)$. Then for $t \geq 1$,
\eqn{
  \PP\!\left( \int_0^t \int_R^\infty u_s(x) dx ds > 0 \right) \leq 48(1+\theta) e^{\theta t} \int e^{-\frac{(x-(R-1))^2}{4t}} u_0(x) dx. 
}
\end{remark}
%
%
\begin{proof}
The first part of \cite[Lemma~2.1]{T1996} yields that if $v_0 \in \SC_{tem}^+$ is supported outside $(-(R+2),\infty)$ then a solution $v$ to \eqref{equ:SPDE} started at $v_0$ satisfies
\eqan{
  \PP\!\left( \int_0^t \int_{-R}^\infty v_s(x) dx ds > 0 \right)
  & = \PP\!\left( \int_0^t \int_{-R}^R v_s(x) dx ds > 0 \right) \\
  & \leq 48(1+\theta) e^{\theta t} \int e^{-\frac{(|x|-(R+1))^2}{4t}} v_0(x) dx \nn\\
  & \leq 48(1+\theta) e^{\theta t} \int_{-\infty}^{-(R+2)} e^{-\frac{(-x-(R+1))^2}{4t}} v_0(x) dx. \nn
}
Set $v_0(\cdot) \equiv u_0(\cdot+2R)$ and use the shift-invariance of solutions to \eqref{equ:SPDE} to get
\eqn{
  \PP\!\left( \int_0^t \int_R^\infty u_s(x) dx ds > 0 \right) 
  \leq 48(1+\theta) e^{\theta t} \int e^{-\frac{(x-(R-1))^2}{4t}} u_0(x) dx
}
as claimed. 
\end{proof}
%
%

The next lemma provides us with a similar statement for $t>0$ arbitrarily small.
%
%
\begin{lemma} \label{LEM:tribe_lemma_2_1}
Let $R>r>q>0$ arbitrarily fixed. If $u_0 \in \SC_{tem}^+$ is supported outside $(-R,R)$ and $u$ is a solution to \eqref{equ:SPDE} starting in $u_0$, then for $t>0$
\eqn{
  \PP\!\left( \int_0^t \int_{-r}^r u_s(x) dx ds > 0 \right) \leq C \left( \theta + (R-r)^{-2} \right) e^{\theta t} \int \left( 1 \wedge \frac{\sqrt{t}}{|x|-R} \right) \exp\!\left( - \frac{(|x|-R)^2}{4t} \right) u_0(x) dx,
}
where $C$ is a constant. If $u_0$ is supported outside $(q,\infty)$, then
\eqn{
  \PP\!\left( \int_0^t \int_r^\infty u_s(x) dx ds > 0 \right) \leq C \left( \theta + (r-q)^{-2} \right) e^{\theta t} \int \left( 1 \wedge \frac{\sqrt{t}}{q-x} \right) \exp\!\left( - \frac{(q-x)^2}{4t} \right) u_0(x) dx. 
}
\end{lemma}
%
%
\begin{proof}
We follow in main parts the reasoning of the proof of \cite[Lemma~2.1]{T1996} and only indicate the changes where necessary. Let $0 \leq \phi_r \in \SC_c^\infty$ such that $\{x: \phi_r(x)>0 \} = (-r,r)$. For $\lambda \geq 0$ let $\xi^\lambda(t,x)$ be the unique non-negative solution of
\eqn{
\lbeq{equ:laplace-xi}
  \xi_t = \xi_{xx} + \theta \xi - (1/2) \xi^2 + \lambda \phi_r, \quad \xi(0,\cdot) \equiv 0.
}
In \cite[(10)]{T1996} it is shown that 
\eqn{
\lbeq{equ:xi-bound}
  \xi^\lambda(t,x) \leq h(x) \equiv 2\theta + 12 (|x|-r)^{-2} \quad \mbox{ for all } |x|>r, \lambda \geq 0.
}
Set $\psi^\lambda(s,x) \equiv \xi^\lambda(t-s,x)$ for $s \in [0,t]$ and $t>0$ fixed. Reason as in the proof of Dawson, Iscoe and Perkins~\cite[Lemma~3.5]{DIP1989}, (3.2.21)--(3.2.24), leaving out (3.2.23) (note that we use $\lambda$ and $\xi, \psi$ where \cite{DIP1989} use $\theta$ and $u, \bar{u}$). To account for the additional term of $\theta \xi$ in \eqref{equ:laplace-xi}, we add this term in the definition of $\xi$ but note that \cite[(3.2.21)]{DIP1989} remains unchanged. Analoguous reasoning then leads to
\eqn{
  \bar{A}\!\left( e^{\theta s}\psi^\lambda(s,x) \right) = -\lambda e^{\theta s} \phi_r(x)
}
and
\eqn{
  \xi^\lambda(t,x) = \lambda \EE_0^x\!\left[ \int_0^t e^{\theta s} \phi_r(B_s) e^{-\int_0^s \xi^\lambda(t-v,B_v) dv} ds \right],
}
where $E_0^x$ denotes expectation with respect to $P_0^x$, the law of Brownian motion starting in $x$. In case $|x|>R>r$, 
\eqn{
  \xi^\lambda(t,x) \leq E_0^x\!\left[ \1_{\{ T_R < t \}} \xi^\lambda(t-T_R,B(T_R)) e^{\theta T_R} \right]
}
follows with $T_R \equiv \inf\{t: |B(t)| \leq R \}$. Next reason as in \cite[(3.2.11)--(3.2.12)]{DIP1989}, using \eqref{equ:xi-bound} and the definition of $T_R$ to obtain
\eqan{
  \xi^\lambda(t,x) 
  &\leq P_0^x(T_R<t) h(R) e^{\theta t}
  \leq 4 e^{\theta t} h(R) P_0^0\!\left( B_1 > \frac{|x|-R}{\sqrt{t}} \right) \\
  &\leq C e^{\theta t} h(R) \left( 1 \wedge \frac{\sqrt{t}}{|x|-R} \right) \exp\!\left( - \frac{(|x|-R)^2}{4t} \right) \nn
}
for some constant $C$.

Recall from \cite[(9)]{T1996} that
\eqn{
  \EE\!\left[ \exp\!\left( -\lambda \int_0^t \langle u_s , \phi_r \rangle ds \right) \right] \geq \exp\!\left(-\big\langle u_0 , \xi^\lambda(t,\cdot) \big\rangle \right).
}
Take $\lambda \rightarrow \infty$ and use $1-e^{-x} \leq x$ for $x \geq 0$ to conclude
\eqn{
  \PP\!\left( \int_0^t \int_{-r}^r u_s(x) dx ds > 0 \right) \leq C h(R) e^{\theta t} \int \left( 1 \wedge \frac{\sqrt{t}}{|x|-R} \right) \exp\!\left( - \frac{(|x|-R)^2}{4t} \right) u_0(x) dx 
}
concluding the proof of the first claim. The second claim follows as in Remark~\ref{RMK:tribe_lemma_2_1} above. 
\end{proof}
%

\paragraph{Acknowledgements.}
I would like to thank my Ph.D. supervisor Ed Perkins for sparking my interest in this research area during my Ph.D. at UBC and for his continued support and helpful feedback. I could not (yet) answer the question he posed to me but the journey continues. By now the journey includes two postdocs, one at EURANDOM, one at the UDE and this is where it got me so far. For the respective phases of the project, I acknowledge the support by an NSERC Discovery Grant, the Netherlands Organisation for Scientific Research (NWO) and by the DFG through the SPP Priority Programme 1590. 
%

\bibliography{bib}
\bibliographystyle{plain}


\end{document}